\documentclass[EJP]{ejpecp} 

\usepackage{caption}
\usepackage{float}
\usepackage{scalerel}

\DeclareMathOperator*{\bsq}{\scalerel*{\square}{\textstyle\bigcup}} 
\DeclareMathOperator*{\sq}{\scalerel*{\square}{\cup}}
\DeclareMathOperator{\Cyl}{Cyl}
\DeclareMathOperator{\law}{Law}
\DeclareMathOperator{\ord}{ord}
\DeclareMathOperator{\len}{len}

\newcommand{\ed}{\ \stackrel{d}{=} \ }
\newcommand{\FF}{\mbox{${\cal F}$}}
\newcommand{\eps}{\varepsilon}
\newcommand{\bZ}{{\bf Z}}
\newcommand{\bX}{{\bf X}}
\newcommand{\bzeta}{\mathbf{\zeta}}
\newcommand{\bY}{{\bf Y}}
\newcommand{\bx}{{\bf x}}
\newcommand{\by}{{\bf y}}
\newcommand{\MC}{mul\-ti\-pli\-ca\-tive co\-a\-le\-scent }
\newcommand{\cvd}{l^2_{\mbox{{\scriptsize $\searrow$}}}}
\newcommand{\IMC}{in\-ter\-act\-ing mul\-ti\-pli\-ca\-tive co\-a\-le\-scent}
\newcommand{\IMCap}{In\-ter\-act\-ing Mul\-ti\-pli\-ca\-tive Co\-a\-le\-scent}
\newcommand{\cvzero}{l_{\mbox{{\scriptsize $\searrow$}}}}
\newcommand{\ba}{{\bf a}}
\newcommand{\bA}{{\bf A}}
\newcommand{\R}{\mathbb{R}}
\newcommand{\I}{\mathbb{I}}
\newcommand{\N}{\mathbb{N}}
\newcommand{\p}{\mathbb{P}}
\newcommand{\E}{\mathbb{E}}
\newcommand{\MCv}{\mathrm{RMM}}
\newcommand{\G}{\mathrm{G}}
\newcommand{\expf}{\mathrm{Exp}^{\infty}(1)}
\newcommand{\glue}{\rho_1^m}
\newcommand{\comp}{{\rm C}(n,p_n,q_n)}

\SHORTTITLE{New critical regime of the SBM}

\TITLE{Stochastic Block Model in a new critical regime and the \IMCap}

\AUTHORS{%
  Vitalii~Konarovskyi\footnote{Fakult\"{a}t f\"{u}r Mathematik und Informatik, Universit\"{a}t Leipzig, Augustusplatz 10, 04109 Leipzig, Germany; Institute of Mathematics of NAS of Ukraine, Tereschenkivska st. 3, 01024 Kiev, Ukraine.
    \EMAIL{konarovskyi@gmail.com}}
  \and 
  Vlada~Limic\footnote{IRMA, UMR 7501 de l'Universit\'e de Strasbourg et du CNRS, 7 rue Ren\'e Descartes, 67084 Strasbourg Cedex, France. \EMAIL{vlada@math.unistra.fr}}}


\AMSSUBJ{60J75; 
  60K35; 
60B12; 
05C80
} 

\SUBMITTED{???} 
\ACCEPTED{???} 

\ARXIVID{2003.10958} 
\HALID{hal-02520742} 

\VOLUME{0}
\YEAR{2021}
\PAPERNUM{0}
\DOI{10.1214/YY-TN}

\ABSTRACT{This work exhibits a novel phase transition for the classical stochastic block model (SBM).
In addition we study the SBM in the corresponding near-critical regime, and find the scaling limit for the component sizes. 
The two-parameter stochastic process arising in the scaling limit, an analogue of the standard Aldous' {\MC}\!, is interesting in its own right.
We name it the {\em (standard) \IMCap}.
To the best of our knowledge, this object has not yet appeared in the literature.}

\begin{document}



\section{Introduction}
The \MC is a process constructed in \cite{Aldous:1997}.
The {\em Aldous' standard  \MC} is the scaling limit of near-critical Erd\H{o}s-R\'enyi graphs. 
The entrance boundary for the \MC was exhibited in \cite{Aldous:1998,Limic:1998:thesis}.

Informally, the  \MC takes values in the space of
collections
of blocks with mass (a number in $(0,\infty)$) and evolves according to the following dynamics:
\begin{equation}
\label{merge}
\begin{array}{c}
\mbox{ each pair of blocks of mass $x$ and $y$ merges at rate $xy$ }\\
\mbox{into a single block of mass $x+y$.}
\end{array}
\end{equation}
We will soon recall its connection to Erd\H{o}s-R\'enyi \cite{Erdos:1960} random graph, viewed in continuous time. 

Erd\H{o}s-R\'enyi-Stepanov (binomial) model is the prototype random graph.
Over the years more complicated random graphs (and networks) have been introduced, either by theoretical or applied mathematicians, and many aspects of them were rigorously studied in the meantime (see for example 
books by Bollob\'as~\cite{Bollobas:2001}, Durrett~\cite{Durrett:2010}, R. van der Hofstad~\cite{Hofstad:2017,Hofstad:2020}, or a survey by Bollob\'as and Riordan~\cite{Bollobas:2013}). 

In particular, Bollob\'as et al.~\cite{Bollobas:2007} also consider a family of random graphs called the {\em finite-type case} (see their Example 4.3) which includes the {\em Stochastic Block Model} (or SBM for short) as a special case. 
Here we study SBM in a novel near critical regime.
It seems plausible that SBM is not the only (but it is the most natural or elementary) family of random graphs exhibiting this phase-transition.

Let us briefly recall the basic SBM model with $m$ classes, frequently denoted by $G(n,p,q)$. 
The issued random graph has the set of vertices divided (in a deterministic way)
into $m$ subsets ({\em  classes} or blocks) of equal size (here we set this size to $n$ and therefore there are $mn$ vertices in total), and a random set of edges, where the edges are drawn independently, and the intra (resp.~inter) class edges are drawn with probability  $p$ (resp.~$q$).
We prefer to use the word ``class" in the present context, since in this area ``block'' is frequently used interchangeably with ``connected component''. 
We consider large graphs ($n$ will diverge to $\infty$), and the connectivity parameters $p$ and $q$ will depend on $n$.
There is one considerable difference in the connectivity parameter scaling (as $n$ diverges) in our work with respect to that in \cite{Bollobas:2007}, applied to finite-type graphs. 
Like in \cite{Bollobas:2007}, here $p_n$ scales inversely proportionally to $n$, but unlike in \cite{Bollobas:2007} our $q_n$ is of much smaller asymptotic order (notably it scales like $n^{-4/3}$).
Regimes where $q_n << p_n$ seem quite natural from the perspective of applications (in view of the formation of clusters in networks).

The scaling limit of Theorem~\ref{the_convergence_of_stochastic_model} in Section~\ref{sec:convergence_of_stochastic_block_model}, is to the best of our knowledge, a completely new stochastic object, of an independent interest to probability theory and applications.
We name it the {\em (standard) \IMC}.
Here two or more (initially independent) multiplicative coalescents interact through an analogue of the ``color and collapse'' mechanism (from \cite{Aldous:1998,Limic:1998:thesis}), which will be made more explicit in our work in progress \cite{Konarovskyi:EIMC:2020}.

The results of our study, proved rigorously for SBM, imply in some (imprecise) sense that for many ``typical'' SBM-alike large neworks there may exist another phase transition (for the formation of the giant component) when the intra-class connectivity is much (an order of magnitude) larger than the inter-class connectivity.
In Section~\ref{sec:concluding_remarks} we provide some further discussion.

We now introduce additional notation necessary for the presentation of our setting and results in the forthcoming sections.
While $l^2$ typically denotes the usual Hilbert
space of all sequences $\bx=(x_k)_{k\geq 1}$ with $\|\bx\|^2:=\sum_{ k=1 }^{ \infty }x_k^2<\infty $, we will henceforth assume without further mention that any $\bx \in l^2$ (of interest to us) also satisfies
$$ 
x_k  \in [0,\infty), \ \forall k\geq 1. 
$$
The subset of $l^2$ consisting of all $\bx$ with non-increasing component will be denoted by $\cvd$. 
Note that $\cvd$ is a Polish space with respect to the induced metric. 
In addition let $\cvzero$ be the set of all infinite vectors with non-increasing components in $[0,+\infty]$ (here the coordinates could take value $\infty$). 
Let $$\ord(\bx): [0,+\infty]^\infty \mapsto \cvzero$$ be the map which non-increasingly orders the components (coordinates) of an infinite vector (in case of ties, the ordering is specified in some natural way, not particularly important for our study). Note that $\ord$ cannot be well-defined (in the sense that all the coordinates of $\bx$ become listed in $\ord(\bx)$) unless for any finite $M$ there are 
at most finitely many coordinates of $\bx$ which are larger than $M$. We implicitly assume this property here, as it
will be true (almost surely) in the sequel.

The \MC is a process taking values in $\cvd$. 
If it starts from an initial value $\bx$ in $l^1$, it will reach the constant state $(\sum_i x_i, 0,\ldots)$ in finite time.
This is not considered an interesting behavior.
If it starts from an initial value $\bx\not \in l^2$, it will immediately collapse to
$(\infty, 0,\ldots)$, which is even less interesting.
Finally, if it starts from $\bx\in \cvd \setminus l^1$, it will stay forever after in this state space.
It is a Markov process with a generator that we prefer to describe in words: if the current state is $\bx$ then for any two different $i,j$ the jump to $\ord((x_i+x_j, \bx^{-i,-j}))$ happens at rate $x_ix_j$. Here $\bx^{-i,-j}$ is the vector obtained from $\bx$ by deleting the $i$th and the $j$th coordinate.

The \MC has interesting entrance laws which live at all times (or rather they are parametrized by $\R$).
The first and the most well-known such law is called the {\em Aldous standard (eternal) \MC}. 
Following the tradition set in \cite{Aldous:1997}, we denote it here by $(\bX^*(t),t\in \R)$.
The law of $\bX^*$ is closely linked (see \cite{Aldous:1997,Armendariz:2001,Broutin:2016}) to that of the following family of diffusions with non-constant drift
$$
\left\{\left\{B_s + ts - \frac{s^2}{2}, \ s \geq 0\right\},\ t\in \R\right\}.
$$
There are uncountably many different non-standard extreme eternal \MC entrance laws and they have been classified in \cite{Aldous:1998}, and linked further in \cite{Martin:2017,Limic:2019} to an analogous family of L\'evy-type processes.
We discuss and use these links in detail in~\cite{Konarovskyi:EIMC:2020}. 


The rest of the paper is organized as follows. 
In Section~\ref{S:RMM} 
we introduce the notion of restricted multiplicative merging $\MCv$, and we also state its basic properties. Section~\ref{sec:continuity_of_the_rmm_in_the_spatial_variable} is devoted to the continuity of $\MCv$ which is used for the proof of the scaling limit of stochastic block model in Section~\ref{sec:convergence_of_stochastic_block_model}. 
A brief discussion of phase transition for the stochastic block model and of the Markov property of the stochastic block model and the {\IMC} is done in Section~\ref{sec:concluding_remarks}. 

\subsection{Restricted multiplicative merging and first consequences}
\label{S:RMM}
A key initial observation in \cite{Aldous:1997} is that Erd\H{o}s-R\'enyi graph, viewed in continuous time, is a (finite-state space) \MC\!\!. 
We extend this construction to fit our purposes in Section \ref{sec:convergence_of_stochastic_block_model}.

For any triangular table (or matrix) $\ba=\{ a_{i,j}: i,j \in \N,\ i<j \}$ of non-negative real numbers, 
any symmetric relation $R$ on $\N$ and any $t \geq 0$ we define the {\em restricted multiplicative merging} (with respect to ($\ba,R$))  
\[
  \MCv_t(\cdot;\ba,R):l^2\to \cvzero 
\]
as follows: to an element $\bx \in l^2$ associate a labelled graph, denoted by $\G_t(\bx;\ba,R)$ 
\begin{itemize}
\item
with
the vertex set $\N$
and the edge set $\{\{i,j\} \in R: a_{i,j}\leq x_i x_j t \}$,
 where it is natural to extend the definition $a_{i,j} := a_{j,i}$ whenever $j<i$;
\item
 for each $i$ assign label $x_i$ to vertex $i$, this label we also call the {\em mass} of $i$;
\item
each connected component (in the usual graph theoretic sense) 
of  $\G_t(\bx;\ba,R)$ is then endowed with its {\em total mass} - the sum of labels or masses $x_\cdot$ over all of its vertices.
\end{itemize} 
Then $\MCv_t(\bx;\ba,R)$ is defined as the vector of the ordered masses of components of $G_t(\bx;\ba,R)$. 
Note that $\MCv_t(\cdot;\ba,R)$ is a deterministic map, and it is clearly measurable
since, for each $n$, $\MCv_t(\cdot;\ba,R):\R^n \to \cvzero$ (defined in the same way as above except with finitely many initial blocks) 
is continuous except on the set $\{\by: \exists i,j \in [n] \mbox{ s.t. } a_{i,j}= y_i y_j t\}$, a union of finitely many demi-hyperbolas (in particular a set of measure $0$), and moreover if $\pi^n(\bx):=(x_1,x_2,\ldots,x_n)$ then 
$\MCv_t(\pi^n(\cdot);\ba,R)$ converges in $\cvd$ to $\MCv_t(\cdot;\ba,R)$ point-wise.
 
Furthermore if $\bA=(A_{i,j})_{i,j}$ is a family of i.i.d.~exponential (rate $1$) random variables, and the relation $R^*$ is maximal (meaning $\{i,j\}\in R^*$ for all $i \neq j$) then $\MCv_t(\bx;\bA,R^*)$ has the law of the \MC  (from the introduction) started from $\ord(\bx)$ and evaluated at time $t$. 
In particular $\p\Big\{\MCv_t(\bx;\bA,R^*) \in \cvd\Big\} =1$. 
Even more is true: $\left\{\MCv_t(\bx;\bA,R^*),\ t \geq 0\right\}$ is equivalent to the Aldous graphical construction of the \MC  process started from $\ord(\bx)$ at time $0$.

It now seems natural to extend the above graphical construction with any infinite relation $R$ as the third parameter.
For our study a family of conveniently chosen relations will be particularly interesting, as explained in Section~\ref{sec:convergence_of_stochastic_block_model}.

For two elements $\bx=(x_k)_{k\geq 1},\by=(y_k)_{k\geq 1} \in l^2$ we will write $\bx\leq \by$ if $x_k\leq y_k$ for every $k\geq 1$. The following lemma is a trivial consequence of the definitions (and the inequality $x^2+y^2\leq (x+y)^2$, $x,y\geq 0$).

\begin{lemma} 
  \label{lem_norm_property_of_mc}
  For any two symmetric relations $R_1\subseteq R_2$, $\bx\leq \by$ and any two times $0\leq t_1\leq t_2$ we have that
  $\G_{t_1}(\bx;\ba,R_1) \subset \G_{t_2}(\bx;\ba,R_2)$ and
  \[
    \left\| \MCv_{t_1}(\bx;\ba,R_1) \right\|\leq \left\| \MCv_{t_2}(\by;\ba,R_2) \right\| ,
  \]
  where we extend the definition of  $\|\cdot \|$ to infinity on $\cvzero \setminus \cvd$.
\end{lemma}

From now on we use notation $$\bA\stackrel{d}{\sim}\expf$$ to indicate a family of i.i.d.~exponential (rate $1$) random variables.

\begin{remark}
If $ \bx  \in \cvzero \setminus \cvd$, $\bA \stackrel{d}{\sim}\expf$, and $t>0$, it is not hard to check that almost surely
$\MCv_t(\bx;\bA,R^*) =(\infty,0,0,\ldots)$, however the above lemma makes sense also for stronger restrictions $R_1$ and $R_2$ where the $l^2$ norm of one or both of the restricted multiplicative mergers is finite.
\end{remark}

We can also observe the following.
\begin{lemma} 
  \label{L:mc_as_map_from_l2_to_l2}
  Let $\bA \stackrel{d}{\sim}\expf$ be defined on a probability space 
  $(\Omega,\FF,\p)$. 
  Then for every $\bx \in l^2$  we have that 
  $$
  \p\left\{\MCv_t(\bx;\bA,R) \in \cvd,\, \forall R \text{ symmetric relation and }  \forall t\geq 0\right\} =1.
  $$
\end{lemma}
\begin{proof}
Apply Lemma \ref{lem_norm_property_of_mc} together with the above made observations about the \MC  graphical construction.
\end{proof}

\begin{remark} 
Note that one cannot strengthen the statement of Lemma \ref{L:mc_as_map_from_l2_to_l2} so that the almost sure event stays universal over $\bx\in l^2$ even for a single relation $R$, if $R$ relates one $i\in \N$ (for example $i=1$) to infinitely many other numbers $i_l,\,l \in \N$. 
Indeed, we could define a random $\bX \in l^2$ as follows: 
$X_1=1$, 
$X_{i_1} = \sum_{m=1}^\infty \frac{1}{m} \I_{\left\{A_{1,i_1} \in \left(1/(m+1), 1/m\right]\right\}},$
and recursively for $l \geq 2$
$$
X_{i_l} = \sum_{m=1}^\infty \frac{1}{m} \I_{\left\{A_{1,i_l} \in \left(\frac{1}{m+1}, \frac{1}{m}\right], 
\frac{1}{m} \not\in \left\{ X_{i_{l-1}}, X_{i_{l-2}},\ldots, X_{i_1} \right\}\right\}}.
$$ 
Note that due to the elementary properties of i.i.d,~exponentials, with probability $1$ for each $m\geq 1$, $1/m$ appears exactly once in the above sequence almost surely, and there are infinitely many zeroes but this we can ignore. 
One can now conclude easily from the definition of $\MCv$ that $\MCv_1(\bX;\bA,R)$ contains a component of infinite mass, almost surely.
\end{remark}

Mimicking the notation of \cite{Limic:1998:thesis}, Section 2.1 
let us denote, for any (think large) $m\in \N$,  $\ba^{[m\uparrow]}:=\{a_{i+m,j+m}\}_{i<j}$, $\bx^{[m\uparrow]}=(x_{k+m})_{k\geq 1}$ and $R^{[m\uparrow]}=\{\{i-m,j-m\}:\, \{i,j \}\in R,\, i,j \geq m+1\}$. 
Note that $R^{[m\uparrow]}$ is not the restriction of  $R$ to $\{m+1,m+2,\ldots \}$, but rather a ``shift of $R$". 

\begin{lemma} 
  \label{lem_removing_of_elements_from_mc}
  For every $m\in\N$, $\bx\in \cvd$, $\ba$ and $R$ as specified above 
  \[
    \left\|\MCv_{t}(\bx;\ba,R)^{[m\uparrow]}\right\|\leq \left\|\MCv_{t}(\bx^{[m\uparrow]};\ba^{[m\uparrow]},R^{[m\uparrow]})\right\|,
  \]
  where we again allow value infinity on both sides of the inequality.
\end{lemma}

\begin{proof} %
Let us first look at how $\G_t\left(\bx^{[m\uparrow]};\ba^{[m\uparrow]},R^{[m\uparrow]}\right)$ 
can be constructed from
$\G_t(\bx;\ba,R)$.
The vertices with masses $x_1,\ldots,x_m$ are removed, as well as any of the edges in $\G_t(\bx;\ba,R)$ connecting any of these vertices to any other vertex.
The other vertices and edges in $\G_t(\bx;\ba,R)$ are kept in $\G_t\left(\bx^{[m\uparrow]};\ba^{[m\uparrow]},R^{[m\uparrow]}\right)$, it is precisely the shift of $R$ that ensures this (deterministic) coupling of the two graphs.

The number of connected components of $\G_t(\bx;\ba,R)$ changed in the just described procedure is $k_m\leq m$.
Let us assume that these components have indices $i_1,\ldots,i_{k_m}$.
On the other hand $\G_t(\bx;\ba,R)^{[m\uparrow]}$ is obtained from $\G_t(\bx;\ba,R)$ via removal of all the (vertices and edges) in the first (and largest) $m$ components of $\G_t(\bx;\ba,R)$.
The claim now follows since for all $j\leq k_m$ the $j$-th largest component of $\G_t(\bx;\ba,R)$ has mass larger or equal to that of its $i_j$-th largest component. 
\end{proof}

\section{Continuity of the RMM in the spatial variable}%
\label{sec:continuity_of_the_rmm_in_the_spatial_variable}

\subsection{Basic notation and formulation of the statement}%
\label{sub:basic_notation_and_formulation_of_the_statement}

In this section, we will state a type of continuity of the map $\MCv$ in the first variable.
This result will be used later for the proof of the SBM scaling limit.

\begin{proposition} 
  \label{pro_continuity_of_mc}
  Let $R$ be any symmetric relation on $\N$, $\bA \stackrel{d}{\sim}\expf$ and sequences $\bx^{(n)}\to \bx$ in $l^2$, $t_n \to t$ in $[0,+\infty)$ as $n \to \infty$. Define
  \begin{align*}
    \bZ^{(n)}(t)&=\MCv_{t}\left(\bx^{(n)};\bA,R\right),\quad n\geq 1,\\
    \bZ(t)&= \MCv_{t}(\bx;\bA,R).
  \end{align*}
  Then $\bZ^{(n)}(t_n) \to \bZ(t)$ in $\cvd$ in probability as $n \to \infty$.
\end{proposition}

\subsection{Auxiliary statements and proof of Proposition~\ref{pro_continuity_of_mc}}%
\label{sub:auxiliary_statements_and_proof_of_proposition_pro_continuity_of_mc}
In this section, a family of i.i.d.~exponential (rate 1) random variables $\bA=(A_{i,j})_{i,j}$ and a symmetric relation $R$ of $\N$ will be fixed.

\begin{lemma} 
  \label{lem_convergence_for_zero_teils}
  Let $\bx^{(n)}=\left(x_k^{(n)}\right)_{k\geq 1}$, $\bZ^{(n)}$, $\bx$, $\bZ$ be defined as in Proposition~\ref{pro_continuity_of_mc}. 
  Assume that there exists $m \in \N$ such that $x_k^{(n)}=0$ for all $k>m$ and $n\geq 1$.
  Then $\bZ^{(n)}(t_n) \to \bZ(t)$ in $\cvd$ a.s. as $n \to \infty$.
\end{lemma}

\begin{proof} %
  The proof trivially follows from the fact that $\p\left\{ A_{i,j}=x_ix_jt \right\}=0$ for all $i,j \in \N$.
 On the complement of $\cup_{i,j}\{ A_{i,j}=x_ix_jt \}$, the graphical construction of $\bZ^{(n)}(t_n)$ clearly converges to that of $\bZ(t)$.
\end{proof}

For two natural numbers $i,j$ we will write $i\sim j$ if the vertices $i$ and $j$ belong to the same connected component of the graph $\G_t(\bx;\bA,R)$. 
Note that $i\sim j$ iff there exists a finite path of edges 
\[
  i=i_0\leftrightarrow i_1 \leftrightarrow \dots \leftrightarrow i_l=j
\]
connecting $i$ and $j$.

Recall that $\|\cdot \|$ denotes the usual $l^2$ norm.
\begin{lemma} 
  \label{lem_estimation_for_probability_of_connection_of_two_enges}
  For every $\bx=(x_k)_{k\geq 1} \in l^2$ and $t \in \left(0,1/\|\bx\|^2\right)$
  \[
    \p\left\{ i \sim j\right\}\leq \frac{ x_ix_jt }{ 1-t \|\bx\|^2 }. 
  \]
\end{lemma}

\begin{proof} %
  The estimate can be obtained (using the inequality $1-e^{-x} \leq x,\ x\geq 0$) as follows
  \begin{align*}
    \p\left\{ i\sim j \right\}&\leq \sum_{ k=1 }^{ \infty } \left( \sum_{ i_1,\dots,i_{k-1} =1}^{\infty} \prod_{ l=1 }^{ k } \p\left\{ A_{i_{l-1},i_l}\leq x_{i_{l-1}}x_{i_l}t \right\} \right)\\
    &= \sum_{ k=1 }^{ \infty } \left( \sum_{ i_1,\dots,i_{k-1}=1 }^{\infty} \prod_{ l=1 }^{ k } \left(1-e^{-x_{i_{l-1}}x_{i_l}t}\right) \right)
  \end{align*} 
  \begin{align*}
    &\leq \sum_{ k=1 }^{ \infty } \left( \sum_{ i_1,\dots,i_{k-1}=1 }^{\infty} \prod_{ l=1 }^{ k } x_{i_{l-1}}x_{i_l}t \right)\leq x_ix_jt\left( \sum_{ k=1 }^{ \infty }  t^{k-1} \|\bx\|^{2k-2}  \right)\\
    &= \frac{ x_ix_jt }{ 1-t\|\bx\|^2 },
  \end{align*}
  where $i_0=i$ and $i_k=j$.
\end{proof}

We next show that the tails of $\bZ^{(n)}=\left( Z^{(n)}_k \right)_{k\geq 1}$ can be uniformly estimated.

\begin{lemma} 
  \label{lem_uniform_estimation_of_tails}
  Let $R$ be a symmetric relation on $\N$, $\bA \stackrel{d}{\sim}\expf$, and assume that $\bx^{(n)} \to \bx$ in $l^2$. Then for every $T>0$ and $\eps>0$ there exists $m \in \N$ such that for every $n\geq 1$ 
  \[
    \E \sup\limits_{ t \in [0,T] }\left\|\MCv_{t}\left(\left(\bx^{(n)}\right)^{[m\uparrow]};\bA^{[m\uparrow]},R^{[m\uparrow]}\right)\right\|^2< \eps.
  \]
\end{lemma}

\begin{proof} %
  We first estimate for $t \in [0,T]$
  \begin{align*}
     \left\|\MCv_{t}\left(\left(\bx^{(n)}\right)^{[m\uparrow]};\bA^{[m\uparrow]},R^{[m\uparrow]}\right)\right\|^2&\leq 
     \left\|\MCv_{T}\left(\left(\bx^{(n)}\right)^{[m\uparrow]};\bA^{[m\uparrow]},R^*\right)\right\|^2,
  \end{align*}
  by Lemma~\ref{lem_norm_property_of_mc}, where $R^*$ denotes the maximal symmetric relation on $\N$.

  Since $\bx^{(n)} \to \bx$ in $l^2$ as $n \to \infty$, 
  \[
  \sup\limits_{ n\geq 1 }\left\|\left(\bx^{(n)}\right)^{[m\uparrow]}\right\| \to 0\quad \mbox{as} \quad m \to \infty.  
  \]
  So, there exists $m \in \N$ such that 
  \[
    \left\|\left(\bx^{(n)}\right)^{[m\uparrow]}\right\|^2+\frac{\left\|\left(\bx^{(n)}\right)^{[m\uparrow]}\right\|^4T}{1-T\left\|\left(\bx^{(n)}\right)^{[m\uparrow]}\right\|^2} <\eps
  \]
  for all $n\geq 1$. 
  We fix $n$ and write $i\sim j$ if vertices $i$ and $j$ belong to the same connected component of the graph $\G_T\left( \left( \bx^{(n)} \right)^{[m\uparrow]};\bA^{[m\uparrow]},R^* \right)$. 
  Then
  \begin{align*}
    \E\Big\|\MCv_{T}\Big(\left(\bx^{(n)}\right)^{[m\uparrow]}&;\bA^{[m\uparrow]},R^*\Big)\Big\|^2= \sum_{ i=m+1 }^{ \infty } \left( x_i^{(n)} \right)^2+2 \sum_{ m<i<j }^{ \infty } x_i^{(n)}x_j^{(n)}\E \I_{\left\{ i\sim j \right\}} 
    \\ 
 \text{(Lemma \ref{lem_estimation_for_probability_of_connection_of_two_enges}) \ } & \leq \sum_{ i=m+1 }^{ \infty } \left( x_i^{(n)} \right)^2+2 \sum_{ m<i<j }^{ \infty } \frac{\left(x_i^{(n)}x_j^{(n)}\right)^2T}{1-T\left\|\left(\bx^{(n)}\right)^{[m\uparrow]}\right\|^2}\\
 &\leq \left\|\left(\bx^{(n)}\right)^{[m\uparrow]}\right\|^2+\frac{\left\|\left(\bx^{(n)}\right)^{[m\uparrow]}\right\|^4T}{1-T\left\|\left(\bx^{(n)}\right)^{[m\uparrow]}\right\|^2}<\eps.
  \end{align*}
\end{proof}

\begin{corollary} 
  \label{cor_uniform_estimation_of_tails_for_z}
  Under the assumptions of Proposition~\ref{pro_continuity_of_mc}, for every $T>0$ and $\eps>0$ there exists $m \in \N$ such that for every $n\geq 1$ 
  \[
    \E \sup\limits_{ t \in [0,T] }\sum_{ k=m+1 }^{ \infty } \left( Z^{(n)}_k(t) \right)^2< \eps.
  \]
\end{corollary}

\begin{proof} %
  The statement follows directly from Lemma~\ref{lem_uniform_estimation_of_tails}
  and Lemma~\ref{lem_removing_of_elements_from_mc}, since
    \begin{align*}
    \sum_{ k=m+1 }^{ \infty } \left( Z^{(n)}_k(t) \right)^2&=  \left\|\MCv_{t}\left(\bx^{(n)};\bA,R\right)^{[m\uparrow]}\right\|^2\\
    &\leq \left\|\MCv_{t}\left(\left(\bx^{(n)}\right)^{[m\uparrow]};\bA^{[m\uparrow]},R^{[m\uparrow]}\right)\right\|^2.
  \end{align*}
\end{proof}

For $m \in \N$ we introduce the notation 
\begin{equation}
\label{E:Rupdown}
  R^{[\downarrow m\uparrow]}=R \cap \left\{ \{i,j\}: i\leq m,\ j>m \right\}
\end{equation}
and 
\[
  \bx^{[\downarrow m]}=\left(x_k \I_{\left\{ k\leq m \right\}}\right)_{k\geq 1}.
\]

We next prove an analog of Lemma~\ref{lem_estimation_for_probability_of_connection_of_two_enges} in a special case where $R=R^{[\downarrow m \uparrow]}$.  
From now on we will write $i\sim_m j$ for two natural numbers $i,j$ if the vertices $i$ and $j$ belong to the same connected component of the graph $\G_t(\bx;\bA,R^{[\downarrow m \uparrow]})$. 
Note that $t$ is omitted from the notation.
For $i,j\leq m$, $i\sim_m j$ if and only if there exists a finite path of edges in $\G_t(\bx;\bA,R^{[\downarrow m \uparrow]})$
\[
  i=i_0\leftrightarrow j_1 \leftrightarrow i_1\leftrightarrow \dots \leftrightarrow j_k\leftrightarrow i_k=j,
\]
 where $i_1,\dots,i_k \leq m$ and $j_1,\dots,j_k>m$ are all different vertices. 

\begin{lemma} 
  \label{lem_estimate_for_probab_of_connected_comp_special_case}
  For every $\bx=(x_k)_{k\geq 1} \in l^2$, $m\geq 1$ and $t \in \left(0,1/(\|\bx^{[\downarrow m]}\|\|\bx^{[m \uparrow]}\|)\right)$
  \[
    \p\left\{ i \sim_m j\right\}\leq \frac{ x_ix_j\, \kappa_{t,i,j}^m\left( \left\|\bx^{[\downarrow m]}\right\|,\left\|\bx^{[m \uparrow]}\right\| \right)}{1-t^2 \left\|\bx^{[\downarrow m]}\right\|^2 \left\|\bx^{[m \uparrow]}\right\|^2},
  \]
  where 
  \[
    \kappa_{t,i,j}^m(x,y)=
    \begin{cases}
      t^2y^2, & \mbox{ if }\ i,j\leq m, \\
      t^2x^2, & \mbox{ if }\ i,j>m,\\
      t, & \mbox{otherwise}.
    \end{cases}
  \]
\end{lemma}

\begin{proof} %
  We only write the details of the proof for the case $i,j\leq m$. Other cases can be shown analogously. As in the proof of Lemma~\ref{lem_estimation_for_probability_of_connection_of_two_enges} we can estimate
  \begin{align*}
    \p\left\{ i\sim_m j \right\}&\leq \sum_{ k=1 }^{ \infty } \left( \sum_{ i_1,\dots,i_{k-1}=1 }^{ m } \sum_{ j_1,\dots,j_{k}=m+1 }^{ \infty }\prod_{ l=1 }^{ k } (x_{i_{l-1}}x_{j_l}^2x_{i_l}t^2) \right)\\
    &= \sum_{ k=1 }^{ \infty } \sum_{ i_1,\dots,i_{k-1}=1 }^{ m } \sum_{ j_1,\dots,j_{k}=m+1 }^{ \infty }
    \!\! x_i x_j t^2 \prod_{l=1}^{k-1} (tx_{i_l}^2) \prod_{l=1}^{k} (tx_{j_l}^2) \\
    &= x_ix_jt^2 \left\|\bx^{[m \uparrow]}\right\|^2\sum_{ k=1 }^{ \infty } \left(t^2\left\|\bx^{[\downarrow m]}\right\|^2\left\|\bx^{[m \uparrow]}\right\|^2\right)^{k-1} \\
    &= \frac{ x_ix_jt^2 \left\|\bx^{[m \uparrow]}\right\|^2 }{ 1-t^2 \left\|\bx^{[\downarrow m]}\right\|^2 \left\|\bx^{[m \uparrow]}\right\|^2 },
  \end{align*}
  where in the second line of the above expressions we used the fact that, for each $k$, $i_0=i$ and ${i_k=j}$.
\end{proof}

\begin{lemma} 
  \label{lem_estimation_of_probability_in_case_of_intermediate_edges}
  For every $\eps \in (0,1]$ and $m \geq 1$
\[
  \p\left\{ \left\|\MCv_{t}\left(\bx;\bA,R^{[\downarrow m \uparrow]}\right)\right\|^2-\left\|\bx^{[\downarrow m]}\right\|^2\geq \eps \right\}\leq \frac{1}{ \eps }\left\|\bx^{[m \uparrow]}\right\|^2\,P_t\left( \left\|\bx^{[\downarrow m ]}\right\|^2 \right),
\]
where $P_t(s)=2+(4t+2t^2)s+2t^2s^2$.
\end{lemma}

\begin{remark}
Clearly $P_t(s)$ is a polynomial of two parameters.
The degree in $t$ is not important for our purposes. However, the fact that the degree in $s$ equals two is reflected in the proof of Proposition~\ref{pro_continuity_of_mc} given below. In particular, the fourth moment estimate of Lemma \ref{lem_finiteness_of_expectation} is necessary for our argument.
\end{remark}
\begin{proof} %
  For simplicity of notation we set $a:=\left\|\bx^{[\downarrow m]}\right\|^2$, $b:=\left\|\bx^{[m \uparrow]}\right\|^2$ and $c:=1/(1-t^2ab)$. We first assume that $t^2ab\leq  \frac{1}{ 2 }$, so that $c\leq 2$. Due to Chebyshev's inequality and Lemma~\ref{lem_estimate_for_probab_of_connected_comp_special_case},
  \begin{align*}
    \p&\left\{ \left\|\MCv_{t}\left(\bx;\bA,R^{[\downarrow m \uparrow]}\right)\right\|^2-\left\|\bx^{[\downarrow m]}\right\|^2\geq \eps\right\}\\
    & \leq 
    \frac{1}{ \eps }\left(\sum_{ k=1 }^{ \infty } x_k^2+\sum_{ i\neq j } x_ix_j\, \p\left\{ i\sim_m j \right\} -\sum_{ k=1 }^{ m } x_k^2\right)
    \\
    &\leq  \frac{1}{ \eps }\left( b +\sum_{ i\not=j} x_i^2 x_j^2 \, \kappa_{t,i,j}^m(\sqrt{a},\sqrt{b})\, c\right)\\
  &\leq \frac{1}{ \eps }\left( b+t^2 bc \sum_{ i,j=1 }^{ m } x_i^2x_j^2 +t^2 ac \sum_{ i,j=m+1 }^{ \infty } x_i^2x_j^2+2tc \sum_{ i=1 }^{ m }  \sum_{ j=m+1 }^{ \infty } x_i^2x_j^2 \right)\\
  &= \frac{1}{ \eps }\left( b +t^2a^2bc+t^2ab^2c+2tabc\right).
  \end{align*}
  If $t^2ab\geq \frac{1}{2}$, then we can estimate the probability by 1. Hence, 
  \begin{align*}
    \p&\left\{ \left\|\MCv_{t}\left(\bx;\bA,R^{[\downarrow m \uparrow]}\right)\right\|^2-\left\|\bx^{[\downarrow m]}\right\|^2\geq \eps\right\}\\
    &\leq \frac{1}{ \eps }\left( b +t^2a^2bc+t^2ab^2c+2tabc\right)\I_{\left\{ t^2ab\leq  \frac{1}{ 2 },\ c \leq 2 \right\}}+\I_{\left\{ t^2ab> \frac{1}{ 2 } \right\}}\\
    &\leq \frac{1}{ \eps }\left( b +2t^2a^2b+b+4tab\right)\I_{\left\{ t^2ab\leq  \frac{1}{ 2 } \right\}}+2t^2ab\I_{\left\{ t^2ab> \frac{1}{ 2 } \right\}}\\
    &\leq \frac{ b }{ \eps }\left( 2+\left( 4t+2t^2\right)a +2t^2a^2 \right).
  \end{align*}
\end{proof}

For $\bx \in \cvd$ we set $\len(\bx)$ to be the index of the last non-zero coordinate in $\bx$ in case it exists, and otherwise $\len(\bx)=+\infty$. Recall (\ref{E:Rupdown}) and note that   
\[
  \left\{ \{i,j\}:\ i\leq \len(\bx),\ j> \len(\bx) \right\}\equiv (R^*)^{[\downarrow \len(\bx) \uparrow]}.
\]
We denote  $\bx \uplus \by=\left(x_k \I_{\left\{ k\leq m \right\}}+y_{k-m}\I_{\left\{ k>m \right\}}\right)_{k\geq 1}$ for $\bx,\by \in l^2$, where $m=\len(\bx)<\infty$. Note that $\bx\uplus \by \not= \by \uplus \bx$ in general.

\begin{lemma} 
  \label{lem_glue_property}
  For $m\geq 1$ we set $\tilde{R}^m=R \cup \left( R^* \right)^{[\downarrow m \uparrow]}$ and
  \begin{align*}
    \tilde{\bZ}^m(t)&=\MCv_{t}\left(\bx;\bA,\tilde{R}^m\right),\\
    \bZ^{\leq m}(t)&=\MCv_{t}\left(\bx^{[\downarrow m]};\bA,R\right)
  \end{align*}
  and 
  \[
    \bZ^{>m}(t)=\MCv_{t}\left(\bx^{[m \uparrow]};\bA^{[m \uparrow]},R^{[m \uparrow]}\right).
  \]
  Let also $\tilde{\bA} \stackrel{d}{\sim}\expf$ be independent of $\bA$. 
  Then
  \[
    \law\left( \tilde{\bZ}^m(t),\bZ^{\leq m}(t),\bZ^{>m}(t) \right)=\law\left(\tilde{\bY}^m(t),\bZ^{\leq m}(t),\bZ^{>m}(t) \right),
  \]
  where
  \begin{equation} 
  \label{equ_process_z}
  \tilde{\bY}^m(t)=\MCv_{t}\left(\bZ^{\leq m}(t) \uplus \bZ^{>m}(t);\tilde{\bA},
  (R^*)^{[\downarrow \len(\bZ^{\leq m}(t) )\uparrow]}\right).
  \end{equation}
\end{lemma}

\begin{proof} %
The claim says that, conditionally on $\bZ^{\leq m}(t),\bZ^{>m}(t)$,
the infinite random vector $\tilde{\bZ}^m(t)$ can be constructed 
via procedure  \eqref{equ_process_z}.
  Its proof follows from properties of the exponential distribution and the definition of $\MCv$.
  We provide three figures to help any interested reader construct a detailed argument.
In the figures there are nine blocks in total and $m$ equals five. 
The general setting (with infinitely many blocks, and arbitrary finite $m$) is analogous. 
  \begin{figure}[H]
      \centering
\captionsetup{width=.8\linewidth}
      \includegraphics[width=115mm]{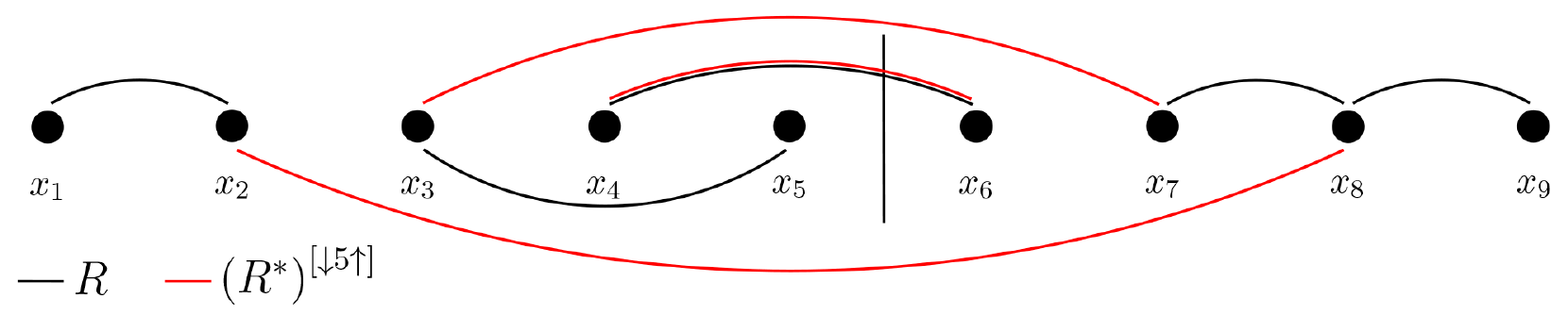}
      \caption{As the legend suggests, the open edges in $R$ are indicated in black, and the
  open edges in $(R^*)^{[\downarrow m \uparrow]}$ are indicated in red.}
  \end{figure}
\begin{figure}[H]
    \centering
    \captionsetup{width=.8\linewidth}
    \includegraphics[width=115mm]{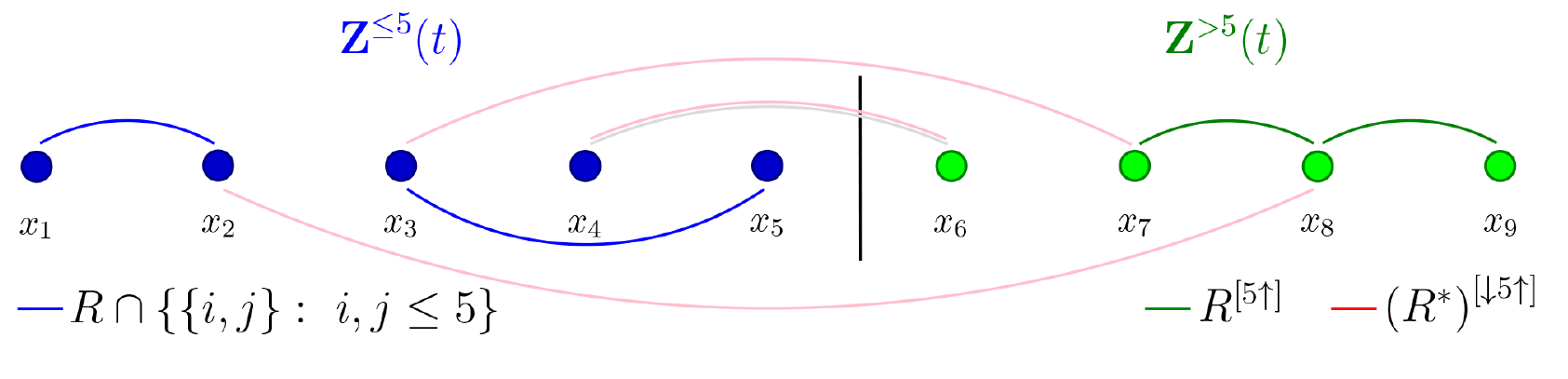}
    \caption{Here the previously red edges are indicated in pink,
  the open edges in $R\setminus (R^*)^{[\downarrow m \uparrow]}$ are indicated in blue or green.}
\end{figure}
\begin{figure}[H]
    \centering
    \captionsetup{width=.8\linewidth}
    \includegraphics[width=115mm]{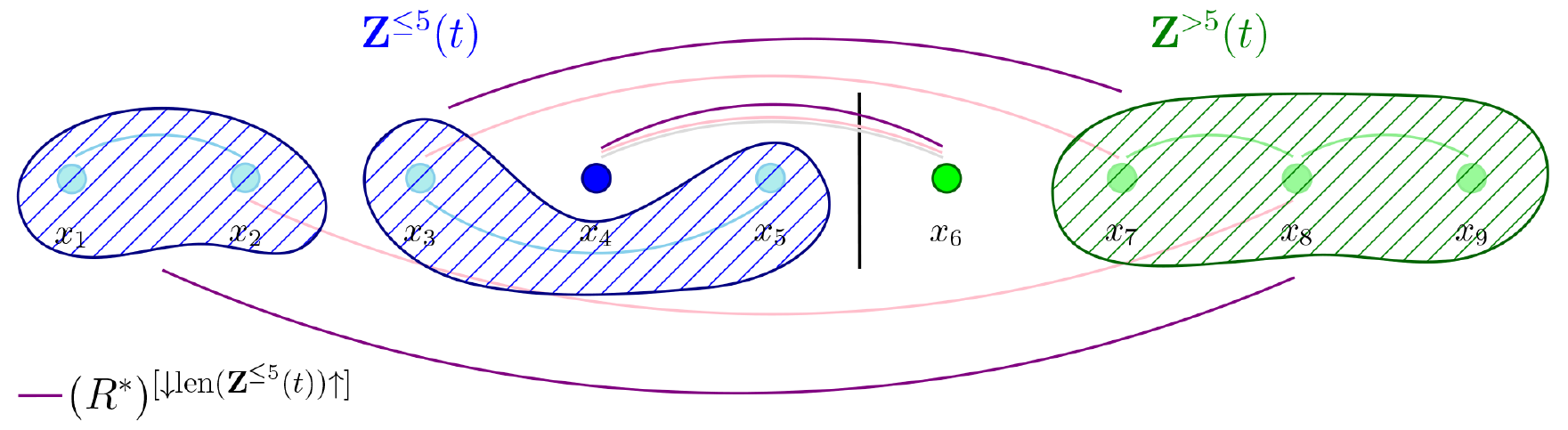}
    \caption{The figure shows the configuration with connected components
  formed based on the (lightly indicated) blue and green open edges.
  The pink (or red) edges can now be superimposed, and combining them results 
  in purple edges between the blocks of $\bZ^{\leq m}$ and $\bZ^{>m}$.}
    \label{fig:my_label}
\end{figure}
\noindent
Note that due to elementary properties of independent exponentials, a purple edges in Figure \ref{fig:my_label} connects the $i$th block of $\bZ^{\leq m}$ and the $j$th block of $\bZ^{>m}$ with probability $1- e^{-t Z_i^{\leq m}Z_j^{>m}}$.
\end{proof}

\begin{lemma} 
  \label{lem_finiteness_of_expectation}
  For every $\bx \in l^2$, $T>0$ and a symmetric relation $R$ on $\N$, one has 
  \[
    \E \sup\limits_{ t \in[0,T] }\left\|\MCv_t\left( \bx;\bA,R \right)\right\|^{4}< +\infty.
  \]
\end{lemma}

\begin{proof} %
  The statement directly follows from Lemma~\ref{lem_norm_property_of_mc} and Theorem~\ref{the_finiteness_of_fourth_moment} in the appendix.
\end{proof}

We now use Lemmas \ref{lem_estimation_of_probability_in_case_of_intermediate_edges} and \ref{lem_glue_property} to obtain the following important uniform bound.

\begin{lemma} 
  \label{lem_estimation_of_probability_of_z_from_its_first_coordinates}
  Let $\bZ(t)=\MCv_{t}\left(\bx;\bA,R\right)$ and $\bZ^{\leq m}(t)$, $\bZ^{> m}(t)$ be as in Lemma~\ref{lem_glue_property}. 
  Then for every $\eps>0$ and $m\geq 1$
  \[
    \p\left\{ \left\|\bZ(t)\right\|^2-\left\|\bZ^{\leq m}(t)\right\|^2\geq \eps \right\}\leq \frac{1}{ \eps }\E \left\|\bZ^{>m}(t)\right\|^2\E\, P_t\left(\left\|\bZ^{\leq m}(t)\right\|^2 \right),
  \]
  where the polynomial $P_t$ is defined in Lemma~\ref{lem_estimation_of_probability_in_case_of_intermediate_edges}.
\end{lemma}

\begin{proof} %
  Let $\tilde{\bZ}^m(t)=\MCv_{t}\left(\bx;\bA,\tilde{R}^m\right)$ and $\tilde{\bY}^m(t)$ be defined by~\eqref{equ_process_z}, where $\tilde{R}^m=R \cup  \left( R^* \right)^{[\downarrow m \uparrow]}$ as before. 
  Due to Lemma~\ref{lem_norm_property_of_mc} we have $ \left\|\bZ(t)\right\|^2 \leq   \|\tilde{\bZ}^m(t)\|^2$ a.s., so 
  \begin{align*}
    \p\Big\{ \left\| \bZ(t) \right\|^2 &-\left\| \bZ^{\leq m}(t) \right\|^2\geq \eps \Big\} \leq \p\left\{ \left\|{\tilde{\bZ}^m(t)}\right\|^2-\left\|\bZ^{\leq m}(t)\right\|^2\geq \eps \right\}\\
   \text{(Lemma \ref{lem_glue_property})\ } & =\p\left\{ \left\|\tilde{\bY}^m(t)\right\|^2-\left\|\bZ^{\leq m}(t)\right\|^2\geq \eps \right\}\\
    &=\E\left( \p\left\{ \left\|\tilde{\bY}^m(t)\right\|^2-\left\|\bZ^{\leq m}(t)\right\|^2\geq \eps \, \bigg| \, \bZ^{\leq m}(t),\bZ^{>m}(t) \right\}\right)\\
    \text{(Lemma \ref{lem_estimation_of_probability_in_case_of_intermediate_edges})\ } & \leq \frac{1}{ \eps }\E \left[\left\|\bZ^{>m}(t)\right\|^2\,P_t\left(\left\|\bZ^{\leq m}(t)\right\|^2 \right)\right]\\
      &= \frac{1}{ \eps }\E \left\|\bZ^{>m}(t)\right\|^2\E\,P_t\left(\left\|\bZ^{\leq m}(t)\right\|^2 \right).
  \end{align*}  
  The final identity follows from the independence of $\bZ^{\leq m}$ and $\bZ^{>m}$.
\end{proof}

\begin{proof}[Proof of Proposition~\ref{pro_continuity_of_mc}] %
  Let $\bx^{(n)}\to \bx$ in $l^2$ and $t_n \to t$ in $[0,+\infty)$ as $n \to \infty$. 
  We recall that $\bZ^{(n)}(t)=\MCv_{t}\left(\bx^{(n)};\bA,R\right)$ and $\bZ(t)=\MCv_{t}\left(\bx;\bA,R\right)$. 
  For $m\geq 1$ we set  
  \begin{align*}
    \bZ^{\leq m}(t)&=\MCv_{t}\left(\bx^{[\downarrow m]};\bA,R\right),\\
    \bZ^{>m}(t)&= \MCv_{t}\left(\bx^{[m \uparrow]};\bA^{[m \uparrow]},R^{[m \uparrow]}\right),\\
    \bZ^{n,\leq m}(t)&=\MCv_{t}\left( \left( \bx^{(n)} \right)^{[\downarrow m]};\bA,R\right),\text{ and }\\ \bZ^{n,>m}(t)&=\MCv_{t}\left( \left( \bx^{(n)} \right)^{[m \uparrow]};\bA^{[m \uparrow]},R^{[m \uparrow]}\right).
  \end{align*}
  
  We fix any subsequence $\{ n_i \}_{i \geq 1}$ of $\N$ and first choose a subsequence $\{ n_{i_l} \}_{l\geq 1}$ of $\{ n_i \}_{i\geq 1}$, denoted by $\{ n'_l \}_{l\geq 1}$ such that $\sum_{ l=1 }^{ \infty } \left\|\bx^{(n'_l)}-\bx\right\|^2<\infty$. We first show that for every $T>0$
   \begin{equation} 
    \label{equ_uniform_boundedness_of_z_m}
    \sup\limits_{ l\geq 1 }\E\sup_m \sup\limits_{ t \in [0,T] } \left\|\bZ^{n'_l,\leq m}(t)\right\|^4<\infty.
  \end{equation}
  Let $\by=(y_k)_{k\geq 1}:=\left( \sup\limits_{ l\geq 1 }\left(\max\left\{ x^{(n'_l)}_k,x_k \right\}\right) \right)_{k\geq 1}$. Then $\by \in l^2$ and $x^{(n'_l)}_k\leq y_k$ for every $k\geq 1$. Indeed, the fact that $\by$ belongs to $l^2$ follows from the estimate 
  \[
    \|\by-\bx\|^2=\sum_{ k=1 }^{ \infty } (y_k-x_k)^2\leq \sum_{ k=1 }^{ \infty } \sum_{ l=1 }^{ \infty } (x^{(n'_l)}_k-x_k)^2=\sum_{ l=1 }^{ \infty } \left\|\bx^{(n'_l)}-\bx\right\|^2<\infty.
  \]
  Hence, by Lemmas~\ref{lem_norm_property_of_mc} and~\ref{lem_finiteness_of_expectation}, we obtain that 
 \begin{equation}
 \label{equ_uniform_boundedness_of_z}
    \sup\limits_{ l\geq 1 }\E\sup\limits_{ t \in [0,T] } \left\|\bZ^{(n'_l)}(t)\right\|^4\leq \E \left\|\MCv_T\left( \by,\bA,R^* \right)\right\|^4<\infty,
  \end{equation}
which implies \eqref{equ_uniform_boundedness_of_z_m} due to Lemma \ref{lem_norm_property_of_mc}.

Let $\eps>0$ be fixed. 
  Due to Lemma~\ref{L:mc_as_map_from_l2_to_l2},
  we have $\lim_{m\to\infty} \left\|\bZ^{\leq m}(t)\right\| = \left\|\bZ(t)\right\|$, and combining this with 
  Lemmas~\ref{lem_uniform_estimation_of_tails},~\ref{lem_estimation_of_probability_in_case_of_intermediate_edges}, ~\ref{lem_estimation_of_probability_of_z_from_its_first_coordinates} and~\eqref{equ_uniform_boundedness_of_z_m}, we can conclude that there exists $m\geq 1$ sufficiently large so that 
  \begin{equation}
  \label{E:tailofZ}
    \p\left\{ \left\|\bZ(t)\right\|^2-\left\|\bZ^{\leq m}(t)\right\|^2\geq \frac{ \eps^2 }{ 9 } \right\}\leq \frac{ \eps }{ 3 },
  \end{equation}
  and
  \begin{equation}
   \label{E:tailofZn}
    \p\left\{ \left\|\bZ^{(n'_l)}(t_{n'_l})\right\|^2-\left\|\bZ^{n'_l,\leq m}(t_{n'_l})\right\|^2\geq \frac{ \eps^2 }{ 9 } \right\}\leq \frac{ \eps }{ 3 }.
  \end{equation}
  Next, by Lemma~\ref{lem_convergence_for_zero_teils}, there exists $L \in \N$ such that for all $l\geq L$
  \[
    \p\left\{ \left\|\bZ^{\leq m}(t)-\bZ^{n'_l,\leq m}(t_{n'_l})\right\|\geq \frac{ \eps }{ 3 } \right\}\leq \frac{ \eps }{ 3 }.
  \]
  We can conclude that for all $l \geq L$
  \begin{align*}
    \p\left\{ \left\|\bZ(t)-\bZ^{(n'_l)}(t_{n'_l})\right\|\geq \eps \right\}&\leq \p\left\{ \left\|\bZ(t)-\bZ^{\leq m}(t)\right\|\geq \frac{ \eps }{ 3 } \right\}\\
    &+\p\left\{ \left\|\bZ^{\leq m}(t)-\bZ^{n'_l,\leq m}(t_{n'_l})\right\|\geq \frac{ \eps }{ 3 } \right\}\\
    &+\p\left\{ \left\|\bZ^{n'_l,\leq m}(t_{n'_l})-\bZ^{(n'_l)}(t_{n'_l})\right\|\geq \frac{ \eps }{ 3 } \right\}\\&\leq \frac{ \eps }{ 3 }+\frac{ \eps }{ 3 }+\frac{ \eps }{ 3 }=\eps,
  \end{align*}
 where we applied (\ref{E:tailofZ},\ref{E:tailofZn}) and \cite{Aldous:1997}, Lemma 17 in order to bound from above the right-hand-side in the first and the third line.
  
  This implies that $\left\|\bZ(t)-\bZ^{(n'_l)}(t_{n'_l})\right\| \to 0$ in probability as $l \to \infty$. So, we have shown that for any subsequence $\{ n_i \}_{i\geq 1}$ there exists a subsubsequence $\{ n_{i_l} \}_{l\geq 1}$ such that $\left\|\bZ(t)-\bZ^{(n'_l)}(t_{n'_l})\right\| \to 0$ in probability as $l \to \infty$. 
 Therefore 
    \[
    \left\|\bZ(t)-\bZ^{(n)}(t_{n})\right\| \to 0 \quad \mbox{in probability as}\ n \to \infty.
  \]
\end{proof}

\section{Scaling limit of near-critical stochastic block models}%
\label{sec:convergence_of_stochastic_block_model}

Let $n,m\geq 2$ be given.
Let $G_{n;p,q}^m$ be the random graph issued from the stochastic block model (SBM) $G_m(n,p,q)$, 
with $m$ classes of size $n$.
The structure of $G_{n;p,q}^m$ was described in the Introduction. 
Recall that the edges are drawn independently at random, and each intra-class edge is present (or open) with probability $p$, while each inter-class edge is present (or open) with probability $q$. 

Here we introduce some additional notation.
Denote the vertices of $G_{n;p,q}^m$ by $[nm]:=\{1,2,\dots,nm\}$, and for each  $l=1,\dots,m$ interpret the subset  
 $B_l=\{l,m+l,2m+l\dots, (n-1)m+l\}$ of $[nm]$ as the $l$-th class.
In addition we define the ``round-robin join'' map $\glue$ from $\cvzero \times \ldots \times \cvzero$ to $l^\infty$ as follows:
for $m$ vectors $\bx^1,\ldots, \bx^m\in \cvzero$ let
$$
\glue(\bx^1,\ldots, \bx^m):=(x_1^1,x_1^2,\ldots, x_1^m,x_2^1,x_2^2,\ldots, x_2^m, x_3^1,x_3^2,\ldots).
$$
Note that $\glue$ is not commutative.

 The main point of $\MCv$ and of the study conducted in Section~\ref{sec:continuity_of_the_rmm_in_the_spatial_variable} is that 
a graphical construction of the connected component sizes of $G_{n;1-e^{-t},1-e^{-u}}^m$ (completely analogous to the one for the \MC as discussed in the introduction) can be conveniently given as follows:\\
 (i) let $\bx=(1,1,\ldots,1,0,0,\ldots)$ where there are exactly $nm$ coordinates equal to $1$,\\
 (ii) for each $l\in {1,\ldots,m}$ define relations 
 \begin{equation} 
  \label{equ_r_intra}
  R_{{\rm intra}}^{(m),n;l}=\{ \{i,j\}:\, \ i,j \in B_l\},
 \end{equation}   
 and also define (note that $i \in B_l$ iff $i \!\! \mod m = l$)
 \begin{equation}
  \label{equ_r_inter}
  R_{{\rm inter}}^{(m)}=\{ \{i,j\}:\ (i - j)\!\!\!\! \mod m  \neq 0\}.
 \end{equation} 
(iii) The three-step procedure 
$$
\MCv_u
\left(\glue\left(\MCv_t\left(\bx;\bA,
R_{{\rm intra}}^{(m),n;1}\right),\ldots,\MCv_t\left(\bx;\bA,R_{{\rm intra}}^{(m),n;m}\right)\right) ; \bA', R_{{\rm inter}}^{(m)}\right),
$$ 
where $\bA,\bA'\stackrel{d}{\sim}\expf$ are independent,
has the law of the connected component sizes of $G_{n;1-e^{-t},1-e^{-u}}^m$.
Indeed, in the first step 
\begin{equation}
\label{E:Rforglue}
{\bf R}^l_t:=\MCv_t\left(\bx;\bA,R_{{\rm intra}}^{(m),n;l}\right) 
\end{equation}
gives the
connected component sizes when all the intra-$B_l$ open edges and no other edges are taken into account.
The second step consists of conveniently assembling the data on all the component sizes of all the $m$ classes (here all the intra-class edges and none of the inter-class edges are accounted for) into a single vector 
\begin{equation}
\label{D:Rt}
{\bf R}_t:= \glue\left({\bf R}^1_t, {\bf R}^2_t \cdots, {\bf R}^m_t\right). 
\end{equation}
In the third step, due to the elementary properties of independent exponentials, applying 
$\MCv_u\left(\cdot; \bA', R_{{\rm inter}}^{(m)}\right)$ to ${\bf R}_t$ gives the connected component sizes when all the open edges between different classes are also  taken into account. This is similar in spirit to the construction in Lemma~\ref{lem_glue_property}.

For $t \in \R$, $u\geq 0$ and $n\geq n_0$ sufficiently large, let $\tilde{\bzeta}^{(n)}(t,u)$ denote the vector of decreasingly ordered component sizes of $G^m_{n; n^{-1}+ tn^{ -4/ 3 }, un^{-4/3}}$. Let also 
\[
  \bzeta^{(n)}(t,u)=n^{-2/3}\tilde{\bzeta}^{(n)}(t,u),\quad t \in \R ,\ \ u\geq 0,\ \ n\geq n_0.
\]
We consider $\bzeta^{(n)}(t,u)$ to be a random element of $\cvd$ (for this we append infinitely many zero entries). 
As discussed above we have
\begin{align}
  \bzeta^{(n)}(t,u)=\MCv_{u_n}\Big(\glue\Big(&\MCv_{t_n}\left(\bx^{(n)};\bA,R_{{\rm intra}}^{(m),n;1}\right), \label{E:zetantu} \\
  &\dots,\MCv_{t_n}\left(\bx^{(n)};\bA,R_{{\rm intra}}^{(m),n;m}\Big)\Big); \bA', R_{{\rm inter}}^{(m)}\right),
 \nonumber
\end{align}
where $\bx^{(n)}=\left(n^{-2/3},n^{-2/3},\dots,n^{-2/3},0,0,\dots\right)$ has exactly $nm$ components equal to $n^{-2/3}$, where $\bA,\bA'\stackrel{d}{\sim}\expf$ are independent, and also
\begin{align*}
  t_n&= -n^{4/3}\ln \left( 1-n^{-1}-tn^{-4/3} \right),\\
  u_n&= -n^{4/3} \ln \left( 1-un^{-4/3} \right)
\end{align*}
for all $n\geq n_0$.
Note that $t_n$ and $u_n$ are chosen according to the identities $1-e^{-t_nn^{-2/3}n^{-2/3}}= \frac{1}{ n }+ \frac{t}{ n^{4/3} }$ and $1-e^{-u_n n^{-2/3}n^{-2/3}}= \frac{u}{ n^{4/3} }$ and that $\bx^{(n)}$ differs from $\bx$, defined above, by the normalization factor $n^{-2/3}$. 
Note that the multipliers in the exponent are compatible with the restricted multiplicative merging, since the mass of the particles is now $n^{-2/3}$.
In the original graphical construction all the masses were equal to $1$, and this is implicit in the expression for the first connectivity parameter (which is equal to $1- e^{-t}$ in the construction comprising \eqref{equ_r_intra}--\eqref{D:Rt}). 

Let $\{\bZ^1(t),\ t \in \R \},\dots,\{\bZ^m(t),\ t \in \R \}$ be independent standard multiplicative coalescents.
We set
\[
  \bZ(t)=\left( \bZ^1(t),\dots,\bZ^m(t) \right),\quad t \in \R, 
\]
and define 
\begin{equation} 
  \label{equ_limit_process_for_stoc_block_model}
  \bzeta(t,u)=\MCv_{u}\left(\glue\left(\bZ(t)\right);\bA,R_{{\rm inter}}^{(m)}\right), \quad t \in \R ,\ \ u\geq 0,
\end{equation}
where $\bA \stackrel{d}{\sim}\expf$ is independent of $\bZ^l$, $l=1,\dots,m$.

\begin{theorem} 
  \label{the_convergence_of_stochastic_model}
  For every $t \in \R $ and $u\geq 0$,
  \[
    \bzeta^{(n)}(t,u)\to \bzeta(t,u) \quad \mbox{as}\ \ n \to \infty 
  \]
  in distribution with respect to the topology on $\cvd$ .
\end{theorem}

\begin{proof} %
  In order to prove the theorem, we are going to use the continuity of $\MCv$, stated in Proposition~\ref{pro_continuity_of_mc},
  together with the fact that the standard \MC arises as the scaling limit of near-critical Erd\H{o}s-R\'enyi graph component sizes. 
  Indeed, it is clear that for each $l$ the (deterministic) intersection of $G^m_{n; n^{-1}+ tn^{ -4/ 3 }, un^{-4/3}}$ with $B_l$, where all the inter-class edges are not taken into account, is a realization from $G(n,n^{-1}+ tn^{ -4/ 3 })$.
  In other words, if $\bx^{(n)}$ equals the vector given immediately below \eqref{E:zetantu}, then for each $l$, the random vector
  $$
    \bZ^{l,(n)}(t):=  \MCv_{t_n}\left(\bx^{(n)};\bA,R_{{\rm intra}}^{(m),n;l}\right),
  $$
  multiplied by $n^{2/3}$, is precisely the ordered listing of component sizes of an Erd\H{o}s-R\'enyi graph on $n$ vertices with near-critical connectivity $p_n =  n^{-1}+ tn^{ -4/ 3 }$.
  Moreover, $\{ \bZ^{l,(n)}\}_{l=1}^m$ is clearly an independent family. 
  Let us abbreviate
  \[
    \bZ^{(n)}(t)=\left( \bZ^{1,(n)}(t),\dots,\bZ^{m,(n)}(t) \right).
  \]
We thus rewrite \eqref{E:zetantu} as
\begin{equation} 
  \label{equ_process_z_n}
  \bzeta^{(n)}(t,u)=\MCv_{u_n}\left(\glue\left(\bZ^{(n)}(t)\right);\bA,R_{{\rm inter}}^{(m)}\right), \quad n \geq n_0.
\end{equation}

It remains to show the right hand side of~\eqref{equ_process_z_n} converges in distribution to $\bzeta(t,u)$ as $n \to \infty$.
As a corollary of Theorem~3~\cite{Aldous:1997} and independence, we have that $\bZ^{(n)}(t) \to \bZ(t)$ in $(l^2)^m$ in distribution as $n \to \infty$.
Since $\rho_1^m$ is continuous map, the convergence in law extends to $\rho_1^m(\bZ^{(n)}(t))$. 
The rest is a standard application of continuity of $\MCv$ operation from Section \ref{sec:continuity_of_the_rmm_in_the_spatial_variable}. 
We include an argument for self-containment. 


By the Skorokhod representation theorem, we can choose a probability space $(\Omega,\FF,\p)$ and a sequence of random elements $\hat{\bZ}^{(n)}(t)$, $n\geq n_0$, and $\hat{\bZ}(t)$ in $\left( l^2 \right)^m$ such that 
$$
     \hat{\bZ}^{(n)}(t) \ed \bZ^{(n)}(t) ,\ n\geq n_0,\quad \hat{\bZ}(t) \ed \bZ(t),
$$
and 
  \[
    \hat{\bZ}^{(n)}(t)\to \hat{\bZ}(t) \quad \mbox{in}\ \ \left( l^2 \right)^m\ \ \mbox{a.s.} \ \ \mbox{as}\ \ n \to \infty.
  \]
  We take $\tilde{\bA} \stackrel{d}{\sim}\expf$ defined on another probability space $(\tilde{\Omega},\tilde{\FF},\tilde{\p})$ and set
  \begin{align*}
    \hat{\bzeta}^{(n)}(\omega,\tilde{\omega},t,u_n)&=\MCv_{u_n}\left(\glue\left(\hat{\bZ}^{(n)}(\omega, t)\right);\tilde{\bA}(\tilde{\omega}),R_{{\rm inter}}^{(m)}\right),\quad (\omega,\tilde{\omega})\in\Omega\times\tilde{\Omega},
\\
    \hat{\bzeta}(\omega,\tilde{\omega},t,u)&=\MCv_{u}\left(\glue\left(\hat{\bZ}(\omega, t)\right);\tilde{\bA}(\tilde{\omega}),R_{{\rm inter}}^{(m)}\right),\quad (\omega,\tilde{\omega})\in\Omega\times\tilde{\Omega}.
  \end{align*}
  Then one can conclude $ \bzeta^{(n)}(t,u_n)\ed \hat{\bzeta}^{(n)}(t,u_n)$, $n\geq n_0$, $\bzeta(t,u)\ed \hat{\bzeta}(t,u)$ and 
  \[
    \hat\bzeta^{(n)}(t,u_n)\to \hat\bzeta(t,u)\quad \mbox{in probability as}\ n \to \infty,
  \]
  with respect to the $l^2$ norm on $\cvd$.
  Indeed, for $\eps>0$, we have
  \begin{align*}
    \p\otimes \tilde{\p}&\left\{ \left\|\hat\bzeta^{(n)}(t,u_n)-\hat\bzeta(t,u)\right\|\geq \eps \right\}\\
  &=  \E\left(\tilde{\p}\left\{ \left\|\hat\bzeta^{(n)}(\omega,t,u_n)-\hat\bzeta(\omega,t,u)\right\|\geq \eps  \right\}\right)\\
  &=  \E\Bigg(\tilde{\p}\Bigg\{ \Bigg\|\MCv_{u_n}\left(\glue\left(\hat{\bZ}^{(n)}(\omega, t)\right);\tilde{\bA},R_{{\rm inter}}^{(m)}\right)\\
  &\quad\quad\quad\quad\quad\quad-\MCv_{u}\left(\glue\left(\hat{\bZ}(\omega, t)\right);\tilde{\bA},R_{{\rm inter}}^{(m)}\right)\Bigg\|\geq \eps  \Bigg\}\Bigg)\to 0,
  \end{align*}
  where the convergence a.s.~of the random sequence (of probabilities) inside the expectation is due to  Proposition~\ref{pro_continuity_of_mc}, and the final conclusion due to the dominated convergence theorem.
  As already explained, this completes the proof of the theorem.
\end{proof}

\section{Concluding remarks}
\label{sec:concluding_remarks}

\subsubsection*{Phase transition of the SBM}

%
%

We recall that if $f$ and $g$ are two sequences, then $f(n) \gg g(n)$ (or equivalently $g(n)\ll f(n)$) means that  $\lim_n g(n)/f(n)=0$, and $f(n) \sim g(n)$ means that $\lim_n g(n)/f(n)=1$. 

Let $\comp$ denote the size of largest component of $G(n,p_n,q_n)$. We can conclude from Theorem~\ref{the_convergence_of_stochastic_model} in the previous section that 
\begin{enumerate}
  \item [(i)] if $p_n-\frac{1}{ n }\sim \frac{t}{ n^{4/3} }$ and $q_n\sim \frac{u}{ n^{4/3} }$, $n \to \infty$, then for all $M \in (0,\infty)$ 
    \begin{equation} 
  \label{equ_existence_of_largest_claster}
     \lim_{ n\to\infty }\p\left\{ n^{-2/3}\comp > M \right\} \in  (0,1);
    \end{equation}
\item [(ii)] 
  if $p_n-\frac{1}{ n }\gg \frac{t}{ n^{4/3} }$ without any assumption on $(q_n)_n$, or  $p_n-\frac{1}{ n }\sim \frac{t}{ n^{4/3} }$, $q_n\gg \frac{u}{ n^4/3 }$, $n \to \infty$, then 
for every $M>0$ 
  \[
    \lim_{ n\to\infty }\p\left\{ n^{-2/3}\comp> M \right\}=1;
  \]

\item [(iii)] 
  if $p_n-\frac{1}{ n }\ll \frac{t}{ n^{4/3} }$, $q_n\sim \frac{u}{ n^{4/3} }$, then
for every $M>0$ 
  \[
    \lim_{ n\to\infty }\p\left\{ n^{-2/3}\comp> M \right\}=0.
  \]
  \end{enumerate}

  We remark that in the case $p_n-\frac{1}{ n }\sim \frac{t}{ n^{4/3} }$, $q_n\ll \frac{u}{ n^{4/3} }$, $n \to \infty$, the scaling limit of the stochastic block model $G(n,p_n,q_n)$ is described by a family on $m$ independent standard multiplicative coalescent without interaction. Hence,~\eqref{equ_existence_of_largest_claster} remains true.

We also have from the pure homogenous graph setting that 
if $p_n - \frac{1}{mn} \sim \frac{t}{ n^{4/3}}$ and 
$q_n - \frac{1}{mn} \sim \frac{t}{ n^{4/3}}$,
then the normalized vector of ordered sizes of connected components of $G(n,p_n,q_n)$ converges to a value of standard multiplicative coalescent at time $m^{4/3}t$. In particular,~\eqref{equ_existence_of_largest_claster} is also satisfied.   

In addition, the main result of Bollob\'as et al.~\cite{Bollobas:2007}, applied to the SBM, says that if $p_n \sim \frac{c}{mn}$ and $q_n \sim \frac{d}{mn}$ and
\begin{enumerate}
  \item [(i)] if $c+(m-1) d>m$, then $ \frac{1}{ n }\comp$ converges to a non-zero number in probability.

  \item [(ii)] if $c+(m-1) d\leq m$, then $ \frac{1}{ n }\comp$ converges to zero in probability.
\end{enumerate}

\subsubsection*{Markov property of the \IMC}
Recall the notation of Section \ref{sec:convergence_of_stochastic_block_model} and in particular the construction resulting in \eqref{D:Rt}.
It should be clear that ${\bf R}_t,\ t \geq 0$, is a Markov process.
However, for any fixed $u>0$ the process
$$
\MCv_u\left({\bf R}_t;\bA', R_{{\rm inter}}^{(m)}\right),\ t\geq 0,
$$
does no longer have the Markov property. 
The main obstacle is in the ``loss of information'' on the class membership once the restricted merging $\MCv_u$ is applied. 
For the same reason, for any fixed $t$, the process 
$$
\MCv_u\left({\bf R}_t;\bA', R_{{\rm inter}}^{(m)}\right),\ u\geq 0,
$$
is no longer Markov.
These observations are made on the discrete level, before passing to the limit.
The same remains true for the \IMC.

Nevertheless, the first process above and its scaling limit, given in Section \ref{sec:convergence_of_stochastic_block_model}, are not far from being Markov (they are hidden Markov), and they are still amenable to analysis. 
In a forthcoming work~\cite{Konarovskyi:EIMC:2020}
we construct an excursion representation of \IMC, analogous to those obtained by \cite{Aldous:1997,Aldous:1998,Armendariz:2001,Broutin:2016,Martin:2017,Limic:2019}, however more complicated, and its complexity increases with $m$.



\appendix 

\section{Appendix}
This auxiliary material is included for reader's benefit. The multiplicative coalescent properties proved below are interested in their own right, and our intention is to obtain their generalizations in a separate work in progress~\cite{Konarovskyi:MMC:2020}.

\subsection{Preliminaries}%
\label{sec:preliminaries}

We rely on the notation introduced above. 
In particular, if $\bx$ is a vector in $l^2$ or $\cvd$, then $\|\bx\|$ is its $l^2$-norm.
We reserve the notation $\bX:=(\bX(t), t\geq 0)$ for any \MC\ process, where its initial state will be clear from the context.
Recall that $\bX(t)=(X_1(t),X_2(t),\ldots)$ where $X_j(t)$ is the size of the $j$th largest component at time $t$.

If $n\in \N$ then $[n]=\{1,2,\ldots,n\}$. Here and below $\bA$  denotes a matrix (or equivalently, a two-parameter family) of i.i.d.~exponential (rate $1$) random variables.

Let $(G_t(\bx;\bA,R^*))_{t,\bx}$ be the family of evolving random graphs on $(\Omega,\FF,\p)$ as constructed in the introduction. We now fix $\bx \in l^2$ and $t>0$, and describe a somewhat different construction of the random graph $G_t(\bx;\bA,R^*)$.

Set $\N^2_<:=\left\{ (i,j):\ i<j,\ i,j \in \N \right\}$ and
\[
  \Omega^0=\{ 0,1 \}^{\N^2_<}.
\]
We also define the product $\sigma$-field $\FF^0=2^{\Omega^0}$ and the product measure
\[
  \p_{\bx,t}^0=\bigotimes_{ i<j } \p_{i,j},
\]
where $\p_{i,j}$ is the law of a Bernoulli random variable with success probability $\p_{i,j}\{1\}=\p\left\{ \bA_{i,j}\leq x_ix_jt \right\}$.
Elementary events from $\Omega^0$ will specify a family of open edges in $G_t(\bx;\bA,R^*)$.
More precisely, a pair of vertices $\{i,j\}$ is connected in $G_t(\bx;\bA,R^*)$ by an edge if and only if $\omega_{i,j}=1$ for $\omega=(\omega_{i,j})_{i<j} \in \Omega^0$. 
In other words, $\p_{\bx,t}^0$ is an ``inhomogeneous percolation process on the complete infinite graph $(\N,\{\{i,j\}: i,j\in \N\})$'' (we include the loops connecting each $i$ to itself on purpose). It is clear that the law of thus obtained random graph $G_t(\bx;\bA,R^*)$ is the same (modulo loops $\{i,i\}$) as the one constructed in the introduction.
Note that $R^*$ is the maximal partition, so these are all graphical constructions of the \MC, equivalent to the Aldous \cite{Aldous:1997} original one.\\
For two $i,j\in \N$ we write $\{i \leftrightarrow j\} = \{\{i,j\} \mbox{ is an edge of } G_t(\bx;\bA,R^*)\}$ and we  may also write it as at $\{\{i,j\} \mbox{ is open}\}$.
We also write $\{i\sim j\}$ for the event that $i$ and $j$ belong to the same connected component of the graph $G_t(x;\bA,R^*)$. Then we have, $\omega$-by-$\omega$, that $i\sim j$ if and only if there exists a finite path of edges 
\[
  i=i_0\leftrightarrow i_1 \leftrightarrow \dots \leftrightarrow i_l=j.
\]

Then one can trivially recognize 
\[
  \p\left\{ i \sim j \right\}=\p_{\bx,t}^0\left\{ i \sim j \right\}.
\]

Part of our argument relies on disjoint occurrence. We follow the notation from ~\cite{Arratia:2018}, since they work on infinite product spaces. We will use an analog of the van den Berg-Kesten inequality~\cite{vandenBerg:1985}, the theorem cited below is an analog of Reimer's theorem~\cite{Reimer:2000}.  
Given a finite family of events $A_k$, $k \in [n]$, from $\FF^0$ we define the event 
\[
\bsq\limits_{k=1}^n A_k=\{A_k,\, k \in [n], \mbox{ jointly occur for disjoint reasons}\}. 
\]
Readers familiar with percolation can skip the next paragraph and continue reading either at Lemma \ref{lem_rkb_inequality} or Section \ref{sec:some_auxiliary_statements}.

Let for $\omega \in \Omega^0$ and $K \subset \N^2_<$
\[
  \Cyl(K,\omega):=\left\{ \bar{\omega}:\ \bar{\omega}_{i,j} =\omega_{i,j},\ (i,j) \in K\right\}.
\]
be the {\it thin cylinder} specified through $K$.
Then the event
\[
  [A]_{K}:=\left\{ \omega:\ \Cyl(K,\omega) \subset A \right\}
\]
is the largest cylinder set contained in $A$, such that it is {\em free in the directions indexed by $K^c$}. Define
\[
  \bsq\limits_{k=1}^n A_k=A_1\sq\dots \sq A_n:=\bigcup_{ J_1,\dots,J_n } [A_1]_{J_1}\cap \dots \cap [A_n]_{J_n},
\]
where the union is taken over finite disjoint subsets $J_k$, $k \in [n]$, of $\N^2_<$.

Let $i_k,j_k \in \N$ and $i_k \not= j_k$, $k \in [n]$. Then we have clearly
\[
  \bsq_{k=1}^n\{ i_k\sim j_k \}=\left\{ \!\!\!
  \begin{array}{r}
  	i_k\sim j_k,\ k \in [n],\ \mbox{via mutually disjoint paths}\!\!\!
  \end{array}
  \right\}.
\]

The following lemma  follows directly from Theorem~11~\cite{Arratia:2018}, but since the events in question are simple (and monotone increasing in $t$) this could be derived directly in a manner analogous to \cite{vandenBerg:1985}.

\begin{lemma} 
  \label{lem_rkb_inequality}
  For any $i_k,j_k \in \N$ and $i_k \not= j_k$, $k \in [n]$, we have
  \[
    \p_{\bx,t}^0\left( \bsq_{k=1}^n \{i_k\sim j_k\} \right)\leq \prod_{ k=1 }^{ n } \p\left(i_k\sim j_k \right).
  \]
\end{lemma}

\subsection{Some auxiliary statements}%
\label{sec:some_auxiliary_statements}


Recall Lemma \ref{lem_estimation_for_probability_of_connection_of_two_enges}.
The goal of this section is to obtain a similar estimate for triples and four-tuples of vertices.

\begin{proposition} 
  \label{pro_estimate_for_connected_componnents}
  There exists a constant $C$ such that for every $\bx=(x_k)_{k\geq 1} \in l^2$ and $t \in (0,1/\|\bx\|^2)$
  \[
    \p\left( i_1\sim i_2 \sim i_3\right)\leq C\frac{ x_{i_1}x_{i_2}x_{i_3}t^{3/2} }{ \left(1-t \|\bx\|^2\right)^3 }
  \]
  and 
  \[
    \p\left( i_1\sim i_2 \sim i_3\sim i_4\right)\leq C\frac{ x_{i_1}x_{i_2}x_{i_3}x_{i_4}t^2 }{ \left(1-t \|\bx\|^2\right)^5 }
  \]
  for distinct natural numbers $i_k$, $k \in [4]$.
\end{proposition}

\begin{remark}
We conjecture analogous estimates for $k$-tuples of vertices.
\end{remark}


\begin{proof}[Proof of Proposition~\ref{pro_estimate_for_connected_componnents}] %
  We will focus on the proof of the second inequality. The proof of the first one is similar and simpler. Let $I:=\{i_1,\dots,i_4\}$ and $I^c=\N \setminus I$. We consider $\{i_1\sim\dots\sim i_4\}$ as an event in the probability space $(\Omega^0,\FF^0,\p^0_{\bx,t})$ and observe that it can be written as a union of the following five events:
  \begin{align*}
    A_1:&= \bigcup_{ \sigma }\bigcup_{ \substack{k,l \in I^c\\ k \not=l} }\{ i_{\sigma(1)}\sim k \}\sq\{ k\sim i_{\sigma(2)} \}\sq\{ i_{\sigma(3)}\sim l \}\sq\{ l\sim i_{\sigma(4)} \}\sq\{ k\sim l \},\\
    A_2:&= \bigcup_{ \sigma }\bigcup_{ k \in I^c } \{ i_{\sigma(1)} \sim i_{\sigma(2)} \}\sq \{ i_{\sigma(1)}\sim k \}\sq\{ i_{\sigma(3)}\sim k \}\sq\{ k\sim i_{\sigma(4)} \},\\
    A_3:&= \bigcup_{ \sigma }\bigcup_{ k \in I^c } \{ i_{\sigma(1)}\sim k \}\sq\{ k\sim i_{\sigma(2)}\}\sq\{ i_{\sigma(3)}\sim k \}\sq \{ k\sim i_{\sigma(4)} \},\\
    A_4:&=\bigcup_{ \sigma } \{ i_{\sigma(1)}\sim i_{\sigma(2)} \}\sq\{ i_{\sigma(1)}\sim i_{\sigma(3)} \}\sq\{ i_{\sigma(1)}\sim i_{\sigma(4)} \},\\
    A_5:&=\bigcup_{ \sigma } \{ i_{\sigma(1)}\sim i_{\sigma(2)} \}\sq\{ i_{\sigma(2)}\sim i_{\sigma(3)} \}\sq\{ i_{\sigma(3)}\sim i_{\sigma(4)} \},
  \end{align*}
  where the unions $\bigcup_{ \sigma }$ are taken over all the permutations $\sigma\in S_4$. The following pictures are the graphical representation of events presented in definition of $A_k$, $k \in [5]$, for $\sigma(i)=i$.
   The long double arrows illustrate connections by paths.
 
  \begin{center}
    \includegraphics[height=37mm]{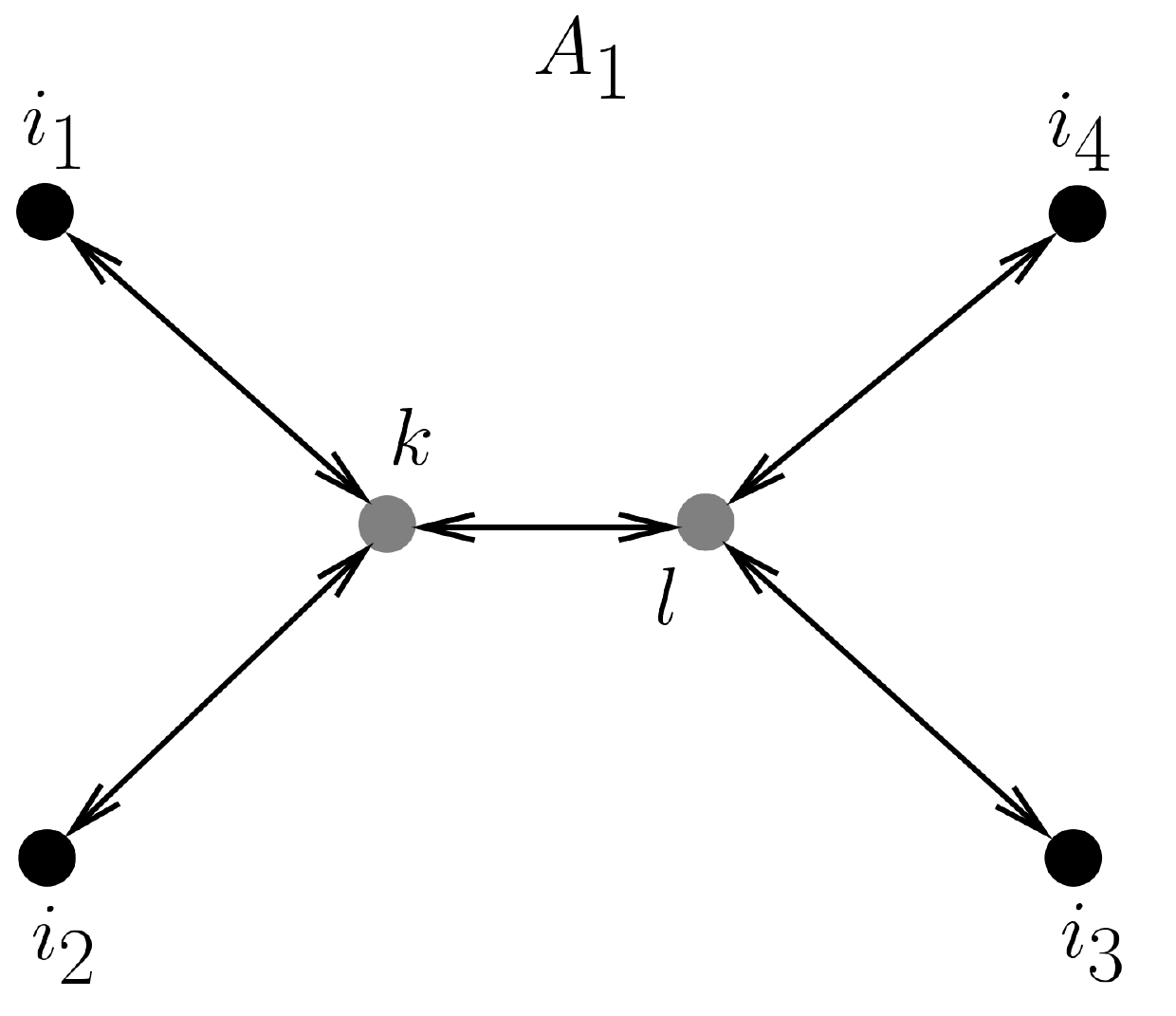}\quad
    \includegraphics[height=37mm]{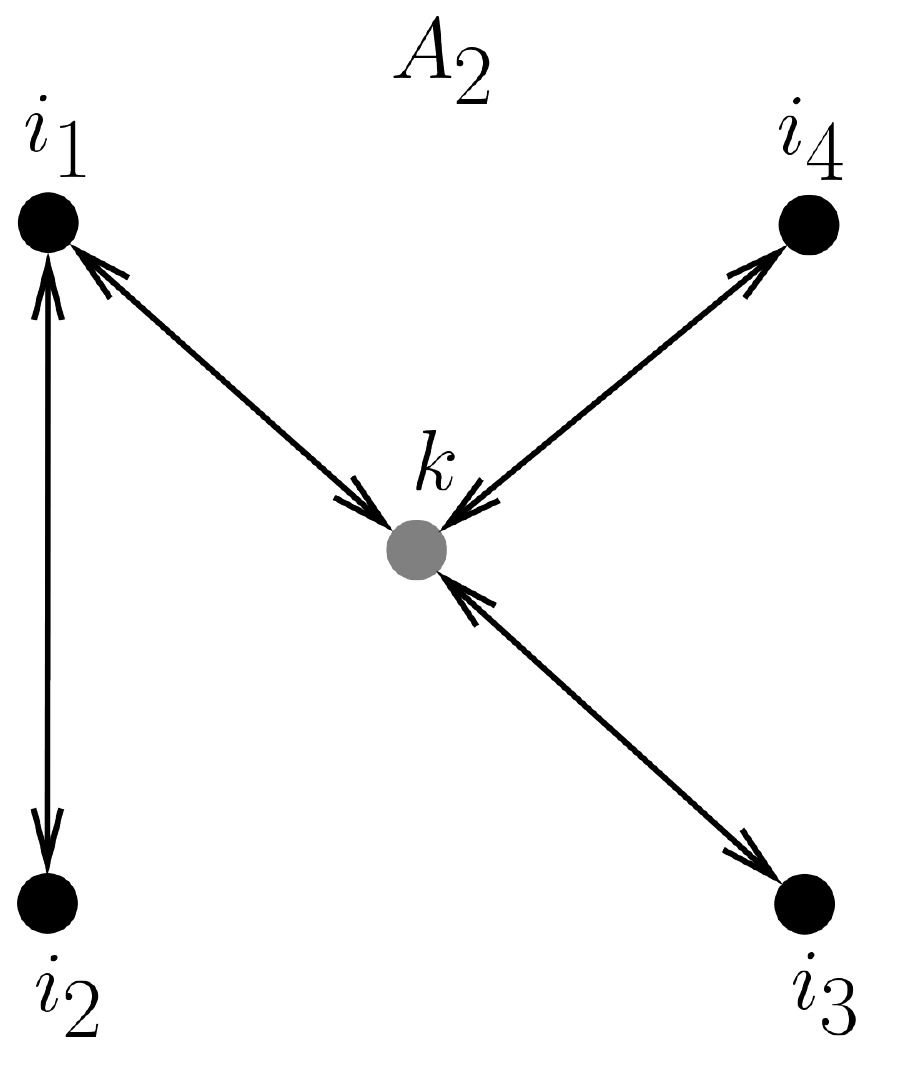}\quad
    \includegraphics[height=37mm]{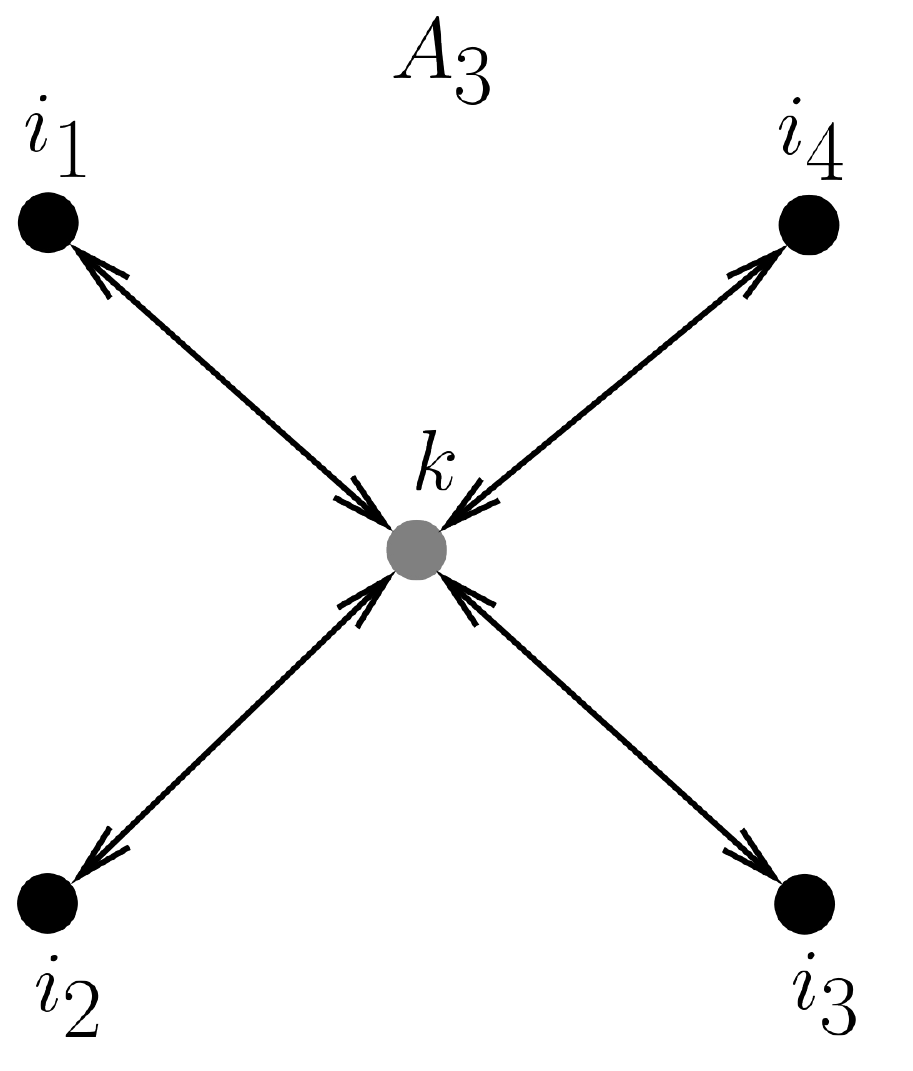}
    \vspace{4mm}

    \includegraphics[height=37mm]{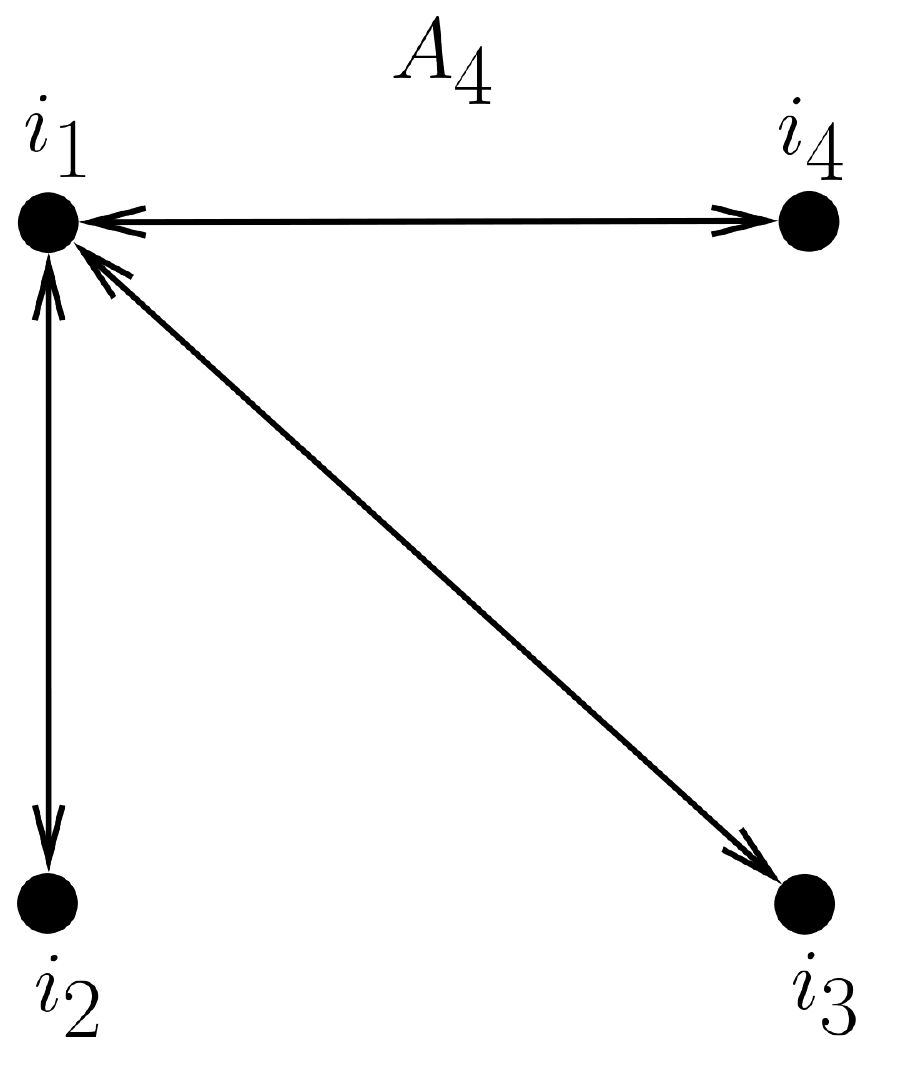}\quad
    \includegraphics[height=37mm]{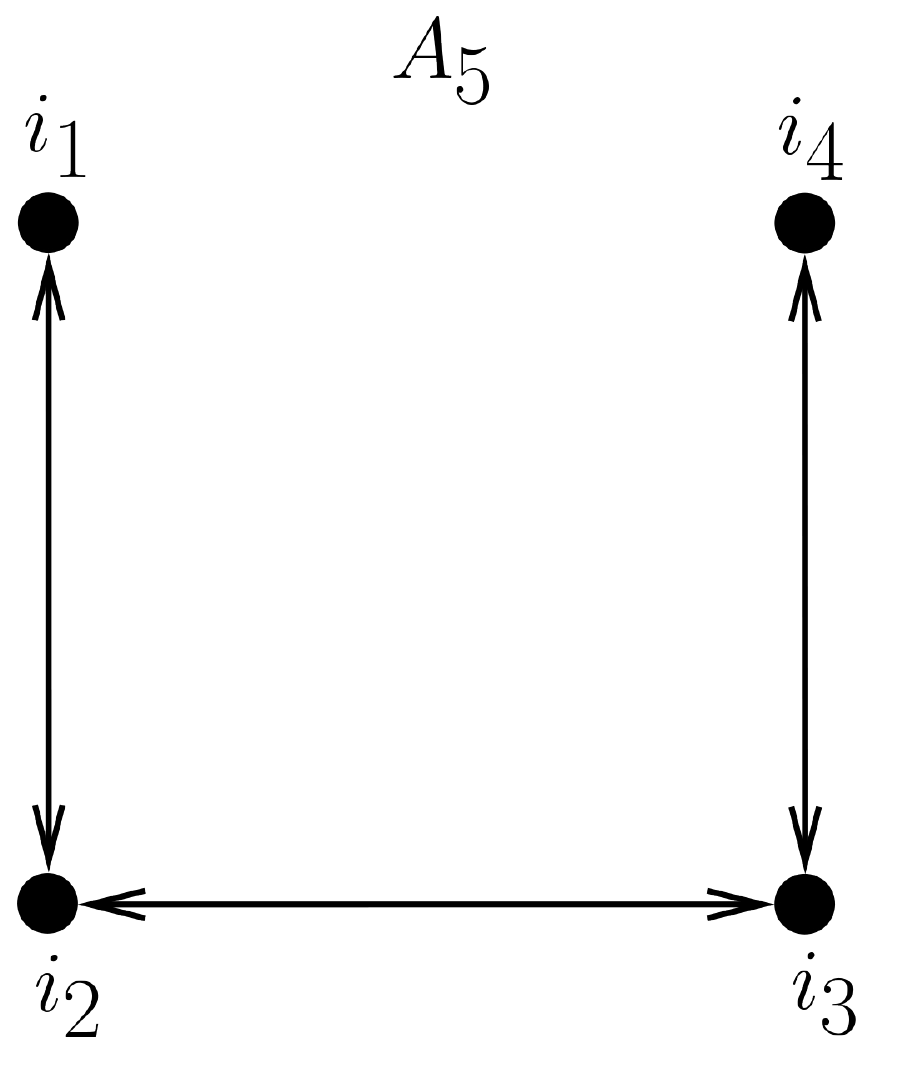}
  \end{center}
  Using Lemma~\ref{lem_rkb_inequality} and the inequalities $t \|\bx\|^2<1$ and \ $\sum_{ k=1 }^{ \infty } x_k^p\leq \|\bx\|^p$ (recall that $\| \cdot\|$ is the norm in $l^2$) for every $p\geq 2$, we can now estimate 
  \begin{align*}
    \p_{\bx,t}^0\left( A_1 \right)\leq \sum_{ \sigma }\frac{ x_{i_1}x_{i_2}x_{i_3}x_{i_4}t^5 }{ (1-t \|\bx\|^2)^5 } \sum_{ k,l=1 }^{ \infty }x_k^3x_l^3 \leq 4!\frac{ x_{i_1}x_{i_2}x_{i_3}x_{i_4}t^5 }{ (1-t \|\bx\|^2)^5 } \|\bx\|^6\leq 4!\frac{ x_{i_1}x_{i_2}x_{i_3}x_{i_4}t^2}{ (1-t \|\bx\|^2)^5 }.
  \end{align*}
  Similarly, we obtain (using $x_{i_{\sigma(j)}} \leq \|\bx \|$ repeatedly)
  \begin{align*}
    \p_{\bx,t}^0\left( A_2 \right)&\leq \sum_{ \sigma }\frac{ x_{i_1}x_{i_2}x_{i_3}x_{i_4}t^4 }{ ( 1-t \|\bx\|^2 )^4 }x_{i_{\sigma(1)}}\sum_{ k=1 }^{ \infty } x_k^3\leq 4!\frac{ x_{i_1}x_{i_2}x_{i_3}x_{i_4}t^4 }{ ( 1-t \|\bx\|^2 )^4 }\|\bx\|^4\leq4!\frac{ x_{i_1}x_{i_2}x_{i_3}x_{i_4}t^2 }{ ( 1-t \|\bx\|^2 )^5 },\\
    \p_{\bx,t}^0\left( A_3 \right)&\leq \sum_{ \sigma }\frac{ x_{i_1}x_{i_2}x_{i_3}x_{i_4}t^4 }{ ( 1-t \|\bx\|^2 )^4 }\sum_{ k=1 }^{ \infty } x_k^4\leq 4!\frac{ x_{i_1}x_{i_2}x_{i_3}x_{i_4}t^4 }{ ( 1-t \|\bx\|^2 )^4 }\|\bx\|^4\leq4!\frac{ x_{i_1}x_{i_2}x_{i_3}x_{i_4}t^2 }{ ( 1-t \|\bx\|^2 )^5 },
    \end{align*} 
    \begin{align*}
    \p_{\bx,t}^0\left( A_4 \right)&\leq \sum_{ \sigma }\frac{ x_{i_1}x_{i_2}x_{i_3}x_{i_4}t^3 }{ ( 1-t \|\bx\|^2 )^3 }x_{i_{\sigma(1)}}^2\leq 4!\frac{ x_{i_1}x_{i_2}x_{i_3}x_{i_4}t^3 }{ ( 1-t \|\bx\|^2 )^3 }\|\bx\|^2\leq4!\frac{ x_{i_1}x_{i_2}x_{i_3}x_{i_4}t^2 }{ ( 1-t \|\bx\|^2 )^5 },\\
    \p_{\bx,t}^0\left( A_5 \right)&\leq \sum_{ \sigma }\frac{ x_{i_1}x_{i_2}x_{i_3}x_{i_4}t^3 }{ ( 1-t \|\bx\|^2 )^3 }x_{i_{\sigma(2)}}x_{i_{\sigma(3)}}\leq 4!\frac{ x_{i_1}x_{i_2}x_{i_3}x_{i_4}t^3 }{ ( 1-t \|\bx\|^2 )^3 }\|\bx\|^2\leq4!\frac{ x_{i_1}x_{i_2}x_{i_3}x_{i_4}t^2 }{ ( 1-t \|\bx\|^2 )^5 }.
  \end{align*}
  Hence, adding over all the five terms above gives
  \[
    \p\left\{ i_1 \sim i_2\sim i_3 \sim i_4 \right\}\leq 5!\frac{ x_{i_1}x_{i_2}x_{i_3}x_{i_4}t^2 }{ ( 1-t \|\bx\|^2 )^5 },
  \]
as stated.
\end{proof}

\subsection{Finiteness of the fourth moment of the multiplicative coalescent}%
\label{sec:finiteness_of_first_moment}

Let $\bX(t)=\MCv_{t}\left(\bx;\bA,R^*\right)$, $t\geq 0$, be a multiplicative coalescent starting from $\bx \in \cvd$. 
The main goal of this section is to prove the following theorem.

\begin{theorem} 
  \label{the_finiteness_of_fourth_moment}
  For every $t\geq 0$ one has 
  \[
    \E \|\bX(t)\|^4< +\infty.
  \]
\end{theorem}

In order to prove the theorem, we first show the finiteness of the fourth moment of the multiplicative coalescent for small $t$ and then extend this result for all $t$.

\begin{lemma} 
  \label{lem_finitenes_of_fourth_norm}
  There exists a constant $C>0$ such that for every $\bx \in l^2$ and $t \in (0,1/\|\bx\|^2)$ the inequality
  \[
    \sum_{ k=1 }^{ \infty } \E X_k^4(t)< \frac{ C\|\bx\|^4 }{ \left(1-t \|\bx\|^2\right)^5 }
  \]
  holds.
\end{lemma}

\begin{proof} %
  For convenience of notation we will here use a natural convention that for each $i$ we have $i\sim i$ almost surely, as indicated in Section \ref{sec:preliminaries}. Using Proposition~\ref{pro_estimate_for_connected_componnents} and the fact $t \|\bx\|^2<1$, we estimate 
  \begin{align*}
    \sum_{ k=1 }^{ \infty } \E X_k^4(t)&\leq \sum_{ i_1,i_2,i_3,i_4=1 }^{ \infty } x_{i_1}x_{i_2}x_{i_3}x_{i_4}\p\left( i_1 \sim i_2 \sim i_3 \sim i_4 \right)\\
    &\leq \sum_{ i=1 }^{ \infty } x_i^4+ 12 \sum_{ i_1 \not= i_2 } x_{i_1}x_{i_2}^3 \p\left( i_1 \sim i_2 \right)+6\sum_{ i_1 \not= i_2 } x_{i_1}^2x_{i_2}^2 \p\left( i_1 \sim i_2  \right)\\
    &+ 12\sum_{ i_1 \not= i_2 \not= i_3 } x_{i_1}x_{i_2}x_{i_3}^2\p\left( i_1 \sim i_2 \sim i_3\right)\\
    &+ \sum_{ i_1 \not= i_2 \not= i_3 \not= i_4 } x_{i_1}x_{i_2}x_{i_3}x_{i_4} \p\left( i_1 \sim i_2 \sim i_3 \sim i_4 \right)\\
    &\leq \|\bx\|^4+ 12 \sum_{ i_1 \not= i_2 } x_{i_1}x_{i_2}^3 \frac{ x_{i_1}x_{i_2}t }{ 1-t \|\bx\|^2 }+6\sum_{ i_1 \not= i_2 } x_{i_1}^2x_{i_2}^2 \frac{ x_{i_1}x_{i_2}t }{ 1-t \|\bx\|^2 }\\
    &+ 12C\sum_{ i_1 \not= i_2 \not= i_3 } x_{i_1}x_{i_2}x_{i_3}^2 \frac{ x_{i_1}x_{i_2}x_{i_3}t^{3/2} }{ \left( 1-t \|\bx\|^2 \right)^3 }\\
    &+ C\sum_{ i_1 \not= i_2 \not= i_3 \not= i_3 } x_{i_1}x_{i_2}x_{i_3}x_{i_4} \frac{ x_{i_1}x_{i_2}x_{i_3}x_{i_4}t^2 }{ \left( 1-t \|\bx\|^2 \right)^5 }\leq \|\bx\|^4+ \frac{ 12\|\bx\|^6 t }{ 1-t \|\bx\|^2 }\\
    &+\frac{ 6\|\bx\|^6t }{ 1-t \|\bx\|^2 }+ \frac{ 12C\|\bx\|^7t^{3/2} }{ \left( 1-t \|\bx\|^2 \right)^3 }+ \frac{ C\|\bx\|^8t^2 }{ \left( 1-t \|\bx\|^2 \right)^5 }\leq \frac{ \tilde{C}\|\bx\|^4 }{ \left( 1-t\|\bx\|^2 \right)^5 }.
  \end{align*}
  This finishes the proof of the lemma.
\end{proof}

Recall that the \MC\ $\bX(t)$, $t\geq 0$, is a Markov process taking values in $\cvd$. Using its generator (in particular, applying it to $\| \bX(t) \|^2$) one concludes that the process 
\begin{equation} 
  \label{equ_martingale}
  M(t):=\|\bX(t)\|^2-\int_{ 0 }^{ t } \left( \|\bX(s)\|^4-\sum_{ k=1 }^{ \infty } X^4_k(s) \right)ds, \quad t\geq 0, 
\end{equation}
is a local martingale (see also equality (68) in~\cite{Aldous:1998}). We will use this fact in order to show the finiteness of the fourth moment of the multiplicative coalescent at small times. 

\begin{proposition} 
  \label{pro_finiteness_of_fourth_moment}
  There exists a constant $C$ such that for every $\bx \in l^2$ and $t \in [0,1/\|\bx\|^2)$
  \begin{equation} 
  \label{equ_expectation_of_integral_of_x}
    \E \int_{ 0 }^{ t } \|\bX(s)\|^4ds\leq \frac{ C \|\bx\|^2 }{ \left( 1-t \|\bx\|^2 \right)^4 }. 
  \end{equation}
  In particular, $\E \|\bX(t)\|^4< +\infty$.
\end{proposition}

\begin{proof} %
  We set 
  \[
    \tau_n:=\inf\left\{ t:\ \|\bX(t)\|\geq n \right\}, \quad n\geq 1.
  \]
  Then $M(t\wedge \tau_n)$, $t\geq 0$, is a martingale for every $n\geq 1$, where $M$ is defined by~\eqref{equ_martingale}. 
  Consequently,
  \[
    \E M(t\wedge \tau_n)=\E \|\bX(t\wedge\tau_n)\|^2-\E\int_{ 0 }^{ t\wedge\tau_n }\left( \|\bX(s)\|^4 - \sum_{ k=1 }^{ \infty } X^4_k(s)\right)ds = \|\bx\|^2
  \]
  for all $n\geq 1$. 
  
  Using Lemma~\ref{lem_finitenes_of_fourth_norm}, the monotonicity of $\|\bX(t)\|$ in $t$ (see for example Lemma~\ref{lem_norm_property_of_mc})  and the estimate for the second moment of the multiplicative coalescent, 
  which can be obtained in a way similar to the proof of Lemma~\ref{lem_finitenes_of_fourth_norm} (see also Lemma~\ref{lem_uniform_estimation_of_tails}), we get
  \begin{align*}
    \E \int_{ 0 }^{ t\wedge\tau_n } \|\bX(s)\|^4ds&= \E \|\bX(t\wedge\tau_n)\|^2-\|\bx\|^2+ \E\int_{ 0 }^{ t\wedge \tau_n }\sum_{ k=1 }^{ \infty } X_k^4(s)ds\\
    &\leq \frac{ C \|\bx\|^2 }{ 1-t \|\bx\|^2 }- \|\bx\|^2+C \int_{ 0 }^{ t } \frac{ \|\bx\|^4 }{ \left( 1-s \|\bx\|^2 \right)^5 }ds \leq \frac{ C \|\bx\|^2 }{ \left( 1-t \|\bx\|^2 \right)^4 }.
  \end{align*}
  By Fatou's lemma, we derive (\ref{equ_expectation_of_integral_of_x}). 
  The finiteness of $\E \|\bX(t(1-\delta))\|^4$ for any small positive $\delta$ now follows again from the monotonicity of $\E \|\bX(t)\|^4$ in $t$, and this in turn implies the stated claim.
\end{proof}

Let $\bA'=(A_{i,j}')_{i,j}$ be an independent copy of $\bA=(A_{i,j})_{i,j}$. 
As in Section \ref{sec:preliminaries} let $(\omega_{i,j}')_{i<j}$ be an independent family of Bernoulli random variables, where $\omega_{i,j}'$ has success probability $\p_{i,j}'\{1\}=\p\left\{ \bA'_{i,j}\leq x_ix_js \right\}$.
We say that ``$\{i,j\}$ is open via $\bA'$'' on the event $\{\omega_{i,j}'=1\}$. 
Note that this does not exclude $\{\omega_{i,j}=1\}$ from happening.
Similarly we say that ``$\{i,j\}$ is open via $\bA$'' on the event
$\{\omega_{i,j}=1\}$.
We will write $i\leftrightarrow_{\bA}j$ whenever $\{i,j\}$ is open via $\bA$.

Denote by $\tilde{G}_{t,s}\left(\bx;\bA,\bA',R^*\right)$ the graph (in fact, it is a multi-graph) constructed by superimposing the edges open via $\bA'$ onto $G_{t}\left(\bx;\bA,R^*\right)$. 
Elementary properties of independent exponentials imply that the vector of ordered component sizes of $\tilde{G}_{t,s}\left(\bx;\bA,\bA',R^*\right)$ is equal in law to $\MCv_{t+s}\left(\bx;\bA,R^*\right)$.

Furthermore, the following property should be clear: if $i<j$,   
\begin{equation}
\label{equ_graph_equal_dist}
\begin{split}
&\mbox{vector of ordered component sizes of }
\tilde{G}_{t,s}\left(\bx;\bA,\bA',R^*\right) \mbox{ given }
\{i\leftrightarrow_{\bA}j\}  \\
&
\quad\quad{\buildrel d \over =}\ \ \mbox{vector of ordered component sizes of }\tilde{G}_{t,s}\left(\bx^{',i,j};\bA,\bA',R^*\right),
\end{split}
\end{equation}
where $\bx^{',i,j}=(x_1,\dots,x_{i-1},x_i+x_j,x_{i+1},\dots,x_{j-1},0,x_{j+1},\dots)$.
Indeed, one could couple the constructions of the two graphs in \eqref{equ_graph_equal_dist} so that 
if $\{k,l\}\cap \{i,j\}=\emptyset$ one uses the exponential thresholds $\bA_{k,l},\bA'_{k,l}$ 
on both sides, 
$(\bA_{k,l})_{\{k,l\}\cap \{i,j\}\neq\emptyset}$ and 
$(\bA'_{k,l})_{\{k,l\}\cap \{i,j\}\neq\emptyset}$ are used only on the left hand side,
while for $l\not\in \{i,j\}$ the ``combined'' thresholds 
$\bar{\bA}_{i,l}:=(x_i+x_j) (\bA_{i,l}/x_i \wedge \bA_{j,l}/x_j)$
and
$\bar{\bA}_{i,l}':=(x_i+x_j) (\bA'_{i,l}/x_i \wedge \bA_{j,l}/x_j)$
(note that these are again exponential (rate $1$) random variables, independent of each other and of $(\bA_{k,l},\bA'_{k,l})_{\{k,l\}\cap \{i,j\}=\emptyset})$ are used on the right hand side. The $j$th component of $\bx^{',i,j}$ is here set to $0$ out of convenience, and it is clear that this ``fake block'' will not contribute to the connected component masses. It is also clear that the exact (coordinate specified) form of $\bx^{',i,j}$ is irrelevant for the statement in \eqref{equ_graph_equal_dist}, the important thing is that the two masses corresponding to $i$ and $j$ are removed from, and another mass of size $x_j+x_j$ is added to, the configuration.
\begin{remark}
In different words, the reasoning above says that one can construct a realization of connected components of
$\tilde{G}_{t,s}\left(\bx;\bA,\bA',R^*\right) \mbox{ given }
\{i\leftrightarrow_{\bA}j\}$ 
from a realization of $\tilde{G}_{t,s}\left(\bx;\bA,\bA',R^*\right)$ by declaring $\{i,j\}$ being open via $\bA$ and keeping all the other $\bA,\bA'$ thresholds, but now the block which (surely) contains both $i$ and $j$ has mass $x_i+x_j$, and the edges via $\bA$ or $\bA'$, which previously separately connected the blocks indexed by $i$ and $j$ to another block indexed by $l$, can (and must) be combined into a single edge which connects the new merger of $i$ and $j$ to $l$. The fact that these combined edges again give rise to multiplicative merging is the key property which makes such processes amenable to analysis. This is not 
the case if the merging mechanism is different (e.g.~exchangeable, additive, or more complicated).
\end{remark}

\begin{proof}[Proof of Theorem~\ref{the_finiteness_of_fourth_moment}] %
  Our argument by contradiction is analogous to ``finite modification'' reasoning in percolation theory.\\
  Let us assume that there exist $t>0$ and $\bx \in l^2$ such that 
  \begin{equation} 
  \label{equ_assumptin_about_divergence}
    \E \left\|\MCv_{t}\left(\bx;\bA,R^*\right)\right\|^4=+\infty.
  \end{equation}
%
  We take $m,M \in \N$ sufficiently large so that the vector
  \begin{equation}
  \label{def_x_grinded}
    \bx^g=\left( \frac{ x_1 }{ M },\dots,\frac{ x_1 }{ M },\frac{ x_2 }{ M },\dots,\frac{ x_2 }{ M },\dots,\frac{ x_m }{ M },\dots,\frac{ x_m }{ M },x_{m+1},x_{m+2}, \dots \right)\!,
  \end{equation}
  obtained by ``grinding'' the first $m$ components (blocks) of $\bx$ each into $M$ new components (blocks) of equal mass, has sufficiently small $l^2$ norm. More precisely, we take $m,M \in \N$ so that
  \[
    t\left\|\bx^g\right\|^2= t\left(\frac{ x_1^2 }{ M }+ \frac{ x_2^2 }{ M }+\dots+\frac{ x_m^2 }{ M }+ x_{m+1}^2+x_{m+2}^2+\dots\right)< \frac{1}{ 2 }.
  \]
  Then $\E \left\|\MCv_{2t}\left(\bx^g;\bA,R^*\right)\right\|^4< +\infty$ due to Proposition~\ref{pro_finiteness_of_fourth_moment}.

  We next consider the event
  \[
    A=\bigcap_{ l=0 }^{ m-1 } \left\{ (lM+1)\leftrightarrow_{\bA}\dots\leftrightarrow_{\bA} (lM+M)\right\}.
  \]
 In words, the grinding done in \eqref{def_x_grinded} is reversed 
 in $G_{t}\left(\bx;\bA,R^*\right)$
 on $A$.
 It is clear that $A$ has positive probability. 
 Other edges may (and typically will) be open but this can only help the chain of inequalities given below.
 
 Using \eqref{equ_graph_equal_dist} and induction we conclude that
\begin{align*}
&\mbox{vector of ordered component sizes of }
\tilde{G}_{t,\frac{t}{2}}\left(\bx^g;\bA,\bA',R^*\right) \mbox{ given }
A\\
&\quad \ {\buildrel d \over =} \ \ \mbox{vector of ordered component sizes of } \tilde{G}_{t,\frac{t}{2}}\left(\bx;\bA,\bA',R^*\right).
\end{align*}
Due to the reasoning of the paragraph above \eqref{equ_graph_equal_dist}, the vector of ordered connected component masses of the graph on the right-hand side is distributed as $\MCv_{\frac{3t}{ 2 }}\left(\bx;\bA,R^*\right)$.
Denote by $\bY$ the vector of order connected component sizes of 
$\tilde{G}_{t,\frac{t}{2}}\left(\bx^g;\bA,\bA',R^*\right)$, and 
observe that $\bY\ {\buildrel d \over =} \ \MCv_{\frac{3t}{ 2 }}\left(\bx^g;\bA,R^*\right) $
for the very same reason.
Therefore,
 \begin{align*}
\infty > \E\left( \left\|\bY \right\|^4\ \right) &\geq 
\E \left[\E\left( \left\|\bY \right\|^4\ \Big|\ \I_A \right) \I_A\right] = 
\E\left[\E \left\|\MCv_{ \frac{ 3t }{ 2 }}\left(\bx;\bA,R^*\right)\right\|^4\I_A\right]
\\
     &=  \p\left( A \right) \E \left\|\MCv_{ \frac{ 3t }{ 2 }}\left(\bx;\bA,R^*\right)\right\|^4 = \infty,
  \end{align*}
a contradiction. 
\end{proof}

\providecommand{\bysame}{\leavevmode\hbox to3em{\hrulefill}\thinspace}
\providecommand{\MR}{\relax\ifhmode\unskip\space\fi MR }
\providecommand{\MRhref}[2]{%
  \href{http://www.ams.org/mathscinet-getitem?mr=#1}{#2}
}
\providecommand{\href}[2]{#2}


\ACKNO{The first author is very grateful to the Mathematics Department at the University of Strasbourg for their hospitality and for providing him with a friendly and stimulating work environment during his visit in February and March 2020. We wish to thank the anonymous referee for their careful reading of the paper, and in particular for spotting an important gap and a mistake in the previous arguments in Section 2.}


\end{document}